\documentclass[11pt,reqno]{amsart}
\usepackage[utf8]{inputenc} % allow utf-8 input
\usepackage[T1]{fontenc}    % use 8-bit T1 fonts
\usepackage{hyperref}
\usepackage{amsfonts}       % blackboard math symbols
\usepackage{microtype}      % microtypography

\usepackage{amscd,amssymb,amsmath,amsthm}
\usepackage[arrow,matrix]{xy}
\usepackage{graphicx}
\usepackage{multicol}
\usepackage{subfigure}
\usepackage{epstopdf}
\usepackage{color}
\newtheorem{thm}{Theorem}
\newtheorem{defn}{Definition}
\newtheorem{lemma}{Lemma}
\newtheorem{pro}{Proposition}

\numberwithin{equation}{section} 

\begin{document}

\vspace{0.5in}
\renewcommand{\bf}{\bfseries}
\renewcommand{\sc}{\scshape}
%insert defs/styles
\vspace{0.5in}

\title[Stability and Bifurcation]{Stability and Bifurcation in a Discrete Phytoplankton-Zooplankton Model with Holling-Type Toxic Effects}

\author{Sobirjon Shoyimardonov}

\address{S.K. Shoyimardonov$^{a,b}$ \begin{itemize}
\item[$^a$] V.I.Romanovskiy Institute of Mathematics, 9, University str.,
Tashkent, 100174, Uzbekistan;
\item[$^b$] National University of Uzbekistan,  4, University str., 100174, Tashkent, Uzbekistan.
\end{itemize}}
\email{shoyimardonov@inbox.ru}

%
%\address{ U.A. Rozikov$^{a,b,c}$\begin{itemize}
%		\item[$^a$] V.I.Romanovskiy Institute of Mathematics,  9, Universitet str., 100174, Tashkent, Uzbekistan;
%		\item[$^b$]  National University of Uzbekistan,  4, Universitet str., 100174, Tashkent, Uzbekistan;
%		\item[$^c$] Karshi State University, 17, Karabag str., 180100, Karshi,  Uzbekistan.
%\end{itemize}}
%\email{rozikovu@yandex.ru}

% Current address (if needed):
%\curraddr{}

%    Information for second author (if needed):
%\author{Author Two}
%\address{}
%\email{}

%    General info
%%%%%%%%%%%%%%%%%%%%%%%%%%%%%%%%%%%%%%%%%%%%%%%%%%%

%                                                                                                                           %
%         Please use the current 2010 Mathematics Subject Classification:             %
%         http://www.ams.org/mathscinet/msc/                                                        %
%         http://www.zentralblatt-math.org/msc/en/                                                 %
%%%%%%%%%%%%%%%%%%%%%%%%%%%%%%%%%%%%%%%%%%%%%%%%%%%



\begin{abstract} In this paper, we investigate a discrete-time phytoplankton–zooplankton model that incorporates a linear predator functional response alongside a Holling-type toxin distribution. Both Holling type II and type III cases are considered, and we derive conditions on the model parameters that guarantee the existence of positive fixed points. We classify all fixed points and analyze their global stability. Furthermore, we establish the occurrence of a Neimark–Sacker bifurcation at the positive fixed point. Theoretical results are supported by numerical simulations, which illustrate the dynamic behavior of the system.
\end{abstract}

\maketitle

\section{Introduction}

Prey-predator models are fundamental tools in the mathematical modeling of ecological systems, offering a framework to understand the dynamic interactions between two biological populations: a prey species, which serves as a food source, and a predator species, which feeds on the prey. The foundational work in this field originates from the classical Lotka-Volterra equations, which use systems of differential equations to describe continuous-time interactions between species. Over time, these models have been refined to incorporate more realistic ecological factors, including carrying capacities, functional responses, time delays, spatial diffusion, and stochastic effects.

In recent years, extensive research has focused on the dynamics of predator-prey systems \cite{Cah, Chen3, Li}. A key component in these models is the functional response, which characterizes the feeding behavior of predators. Among the most widely studied are the Holling type II and type III functional responses \cite{Holling}, both of which can profoundly influence system dynamics \cite{Ko, Pen, SH, Wang, Zhou}.

While continuous-time models have historically provided critical insights into population dynamics, discrete-time models have gained prominence, especially in systems where population changes occur at distinct intervals—such as seasonal reproduction or data collection at regular time steps. Discrete models can reveal complex behaviors, including periodic oscillations, stability shifts, and chaos, arising from nonlinear species interactions \cite{De, May}.

A notable example of a prey-predator relationship is the interaction between phytoplankton and zooplankton. These organisms form the foundational and intermediate levels of aquatic food webs, respectively, and play essential roles in nutrient cycling, oxygen production, and overall ecosystem functioning. Understanding their dynamics is crucial for assessing ecosystem health, predicting responses to environmental stress, and managing aquatic resources. Recent studies have explored various aspects of plankton dynamics and their ecological implications \cite{Chatt, Chen, Shang, Hong, Mac, Tian, Sajan, RSH, RSHV, SH}.

A central objective in studying ecological models is to determine the long-term behavior of interacting populations. Stability analysis of fixed points helps predict whether a system will reach equilibrium or exhibit sustained oscillations. However, nonlinear ecological systems can also undergo bifurcations-qualitative changes in dynamics triggered by parameter variations. Among these, the Neimark-Sacker bifurcation--the discrete analog of the Hopf bifurcation--marks the transition from fixed-point stability to quasiperiodic or cyclic behavior on invariant closed curves \cite{Kuz}. Detecting such bifurcations is critical, as they signal the onset of complex, and potentially chaotic, ecological dynamics.

In this study, we investigate a discrete-time phytoplankton-zooplankton model that incorporates toxin-mediated interactions and employs both Holling type II and type III functional responses. We derive conditions for the existence and stability of positive fixed points, analyze the role of toxic effects in shaping system dynamics, and identify parameter regimes where Neimark-Sacker bifurcations arise. Our theoretical findings are supported by numerical simulations that demonstrate a rich spectrum of behaviors, from stable equilibria to bifurcation-induced oscillations, thereby deepening our understanding of how toxic interactions and functional responses influence aquatic ecosystems.

The following continuous-time phytoplankton-zooplankton model was introduced in \cite{Chatt}:
\begin{equation}\label{chat}
\left\{
\begin{aligned}
&\frac{dP}{dt}=bP\left(1-\frac{P}{K}\right)-\alpha f(P)Z,\\
&\frac{dZ}{dt}=\beta f(P)Z - rZ - \theta g(P)Z.
\end{aligned}
\right.
\end{equation}
Here, \( P \) and \( Z \) denote the densities of phytoplankton and zooplankton, respectively. The parameters \( \alpha > 0 \) and \( \beta > 0 \) represent, respectively, the predation rate and the conversion efficiency of zooplankton feeding on phytoplankton. The parameter \( b > 0 \) denotes the intrinsic growth rate of phytoplankton, while \( K > 0 \) is the environmental carrying capacity. The zooplankton population experiences a natural mortality rate denoted by \( r > 0 \). The function \( f(P) \) characterizes the predator's functional response to phytoplankton density, and \( g(P) \) describes the distribution of toxic substances released by phytoplankton. The parameter \( \theta > 0 \) quantifies the rate of toxin liberation.

The authors of \cite{Chatt} investigated the local stability of system~\eqref{chat} under various functional forms of \( f(P) \) and \( g(P) \), demonstrating that toxin-producing phytoplankton can play a regulatory role in population dynamics and effectively inhibit uncontrolled plankton blooms.

In \cite{Chen}, the authors investigated system~\eqref{chat} in continuous time by adopting specific functional forms for the functional response and toxin release:
\[
f(P) = \frac{P^h}{1 + cP^h}, \quad g(P) = P, \quad \text{for } h = 1, 2.
\]
%In \cite{SH}, the model was studied in discrete time with the choice
%\[
%f(P) = \frac{P}{1 + cP},
%\]
%and it was shown that a Neimark–Sacker bifurcation can occur at a positive fixed point.

In this paper, we study the system \eqref{chat} with the functional choices
\[
f(P) = P, \quad g(P) = \frac{P^h}{1 + cP^h}, \quad \text{for } h = 1,2.
\]
Using appropriate scaling transformations, the model is reformulated as:

\begin{equation}\label{chenn}
\left\{
\begin{aligned}
&\frac{du}{dt} = u(1 - u) - uv,\\
&\frac{dv}{dt} = \beta uv - rv - \frac{\theta u^h v}{1 + c u^h},
\end{aligned}
\right.
\end{equation}
where $u$ and $v$ represent the rescaled phytoplankton and zooplankton populations, respectively.

In this paper, we extend the analysis of system~\eqref{chenn} by considering its discrete-time counterpart, incorporating both Holling type II and Holling type III toxin distribution functions. Our study focuses on the existence and stability of positive fixed points, as well as the occurrence of Neimark–Sacker bifurcations, which play a crucial role in understanding the system's long-term dynamics.

The paper is organized as follows: In Section~2, we examine the existence and stability of positive fixed points for both the Holling type II and Holling type III cases. Section~3 is devoted to the bifurcation analysis of system~\eqref{h12}, where we prove that a Neimark–Sacker bifurcation occurs at the unique positive fixed point of the model with Holling type II, and at one of the (up to three) positive fixed points of the model with Holling type III. Finally, in Section~4, we present numerical simulations to illustrate and support our theoretical findings.

%$$\overline{t}=bt, \, \overline{u}=\frac{P}{k}, \, \overline{v}=\frac{\alpha Z}{b}, \, \overline{c}={ck^h}, \, \overline{\beta}=\frac{\beta k}{b}, \, \overline{r}=\frac{r}{b}, \, \overline{\theta}=\frac{\theta k^h}{b}.$$

\section{Existence and stability of positive fixed points}

We consider the discrete-time version of model~\eqref{chenn}, which takes the form:
\begin{equation}\label{h12}
V:
\begin{cases}
u^{(1)} = u(2 - u) - uv, \\[2mm]
v^{(1)} = \beta uv + (1 - r)v - \dfrac{\theta u^{h}v}{1 + cu^{h}},
\end{cases}
\end{equation}
where \( u \) and \( v \) represent the densities of phytoplankton and zooplankton, respectively, and all parameters \( \beta, r, \theta, c \) are assumed to be positive throughout the paper.

It is straightforward to verify that system~\eqref{h12} always admits two nonnegative equilibria: \( (0, 0) \) and \( (1, 0) \). To find positive interior fixed points, we set \( v^{(1)} = v \) and obtain that a point \( (u, v) \), with \( u \in (0, 1) \), is a positive fixed point of system~\eqref{h12} if and only if \( u\) satisfies the following equation:

\begin{equation}\label{theta}
\theta=\Psi_h(u):=\frac{(\beta u-r)(1+c u^{h})}{u^{h}}.
\end{equation}

Thus, the operator (\ref{h12}) does not have any positive fixed point when $\beta\leq r.$

It is obvious that:

%1. there exists $\overline{u}>0$ such that $\Psi'(\overline{u})=0,$ $\Psi'(u)>0$ for $u\in(\overline{u},\infty),$  and $\Psi'(u)<0$ for $u\in(0,\overline{u});$

 $\lim_{u\rightarrow0}\Psi_h(u)=\infty$ and $\lim_{u\rightarrow1}\Psi_h(u)=(1+c)(\beta-r).$

\medskip
%In \cite{Chen} given the following useful lemmas.
%
%\begin{lemma} (\cite{Chen}) Assume that $\overline{u}\geq1.$ Then system (\ref{h12}) has no positive fixed point for $\beta<(c+1)(r+\theta),$ and a unique positive fixed point for $\beta>(c+1)(r+\theta).$ \label{lemma1}
%\end{lemma}
%
%\begin{lemma} (\cite{Chen}) Assume that $\overline{u}<1.$ Then system (\ref{h12}) has no positive fixed point for $\beta<\Psi(\overline{u}),$  a unique positive fixed point for $\beta>(c+1)(r+\theta),$  and two positive fixed points for $\beta\in(\Psi(\overline{u}), (c+1)(r+\theta)).$ \label{lemma2}
%\end{lemma}
\begin{pro}\label{prop1}
For the fixed points \( (0, 0) \) and \( (1, 0) \) of system~\eqref{h12}, the following statements hold:
\[
(0, 0) =
\begin{cases}
\text{nonhyperbolic}, & \text{if } r = 2, \\[2mm]
\text{saddle}, & \text{if } 0 < r < 2, \\[2mm]
\text{repelling}, & \text{if } r > 2,
\end{cases}
\]
\[
(1, 0) =
\begin{cases}
\text{nonhyperbolic}, & \text{if } \theta = (\beta - r)(1 + c) \text{ or } \theta = (2 + \beta - r)(1 + c), \\[2mm]
\text{attractive}, & \text{if } (\beta - r)(1 + c) < \theta < (2 + \beta - r)(1 + c), \\[2mm]
\text{saddle}, & \text{otherwise}.
\end{cases}
\]
\end{pro}

\begin{proof} The Jacobian of the operator \eqref{h12} is given by
\begin{equation}\label{jac}
J(u,v) =
\begin{bmatrix}
2 - 2u - v & -u \\
\beta v - \frac{\theta h u^{h-1} v}{(1+cu^h)^2} & \beta u + 1 - r - \frac{\theta u^h}{1+cu^h}
\end{bmatrix}.
\end{equation}

Evaluating at the point \( (0,0) \), we get
\[
J(0,0) =
\begin{bmatrix}
2 & 0 \\
0 & 1-r
\end{bmatrix},
\]
with eigenvalues \( \lambda_1 = 2 \) and \( \lambda_2 = 1-r \). Using this, the proof for \( (0,0) \) follows easily.

Similarly, at \( (1,0) \), the Jacobian becomes
\[
J(1,0) =
\begin{bmatrix}
0 & -\frac{1}{1+c} \\
0 & \beta + 1 - r - \frac{\theta}{1+c}
\end{bmatrix},
\]
with eigenvalues \( \lambda_1 = 0 \) and \( \lambda_2 = \beta + 1 - r - \frac{\theta}{1+c} \).

By solving the conditions \( |\lambda_2| = 1 \) and \( |\lambda_2| < 1 \), we derive the necessary parameter constraints. Therefore, the proposition is proved.

\end{proof}

In this section we consider existence and stability of fixed points for both operators.

 Let $h=1,$  then the operator (\ref{h12})  has the following form
\begin{equation}\label{h1}
V_1:
\begin{cases}
u^{(1)}=u(2-u)-uv\\[2mm]
v^{(1)}=\beta uv+(1-r)v-\frac{\theta uv}{1+cu}.
\end{cases}
\end{equation}

If $h=2,$ then  the operator (\ref{h12})  has the form

\begin{equation}\label{h2}
V_2:
\begin{cases}
u^{(1)}=u(2-u)-uv\\[2mm]
v^{(1)}=\beta uv+(1-r)v-\frac{\theta u^2v}{1+cu^2}.
\end{cases}
\end{equation}

Let $\widetilde{E}=(\widetilde{u},\widetilde{v})$ be a positive fixed point of the operator (\ref{h12}.) Then using $\widetilde{v}=1-\widetilde{u}$ and $\Psi_h(\widetilde{u})=\theta$ we get that the characteristic polynomial of Jacobian (\ref{jac}) at the positive fixed point $\widetilde{E}$ is given by

\begin{equation}\label{chareq}
F(\lambda, \widetilde{u}) = \lambda^2 - p(\widetilde{u})\lambda + q(\widetilde{u}),
\end{equation}
where
\begin{equation}\label{pq}
p(\widetilde{u}) = 2-\widetilde{u}, \quad q(\widetilde{u}) = 1-\widetilde{u}+\widetilde{u}(1-\widetilde{u})\left(\beta-\frac{\theta h\widetilde{u}^{h-1}}{(1+c\widetilde{u}^h)^2}\right).
\end{equation}

\subsection{Existence and stability of positive fixed point of (\ref{h1})}

\begin{pro}\label{prop2} The operator~\eqref{h1} has a unique positive fixed point \(\overline{E} = (\overline{u}, \overline{v})\) if and only if \(\beta > r\) and \(0 < \theta < (\beta - r)(1+c),\) where
\[
\overline{u} = \frac{rc+\theta-\beta+\sqrt{(\beta-rc-\theta)^2+4rc\beta}}{2c\beta}, \quad \overline{v} = 1-\overline{u}.
\]
\end{pro}

\begin{proof} Let \( h=1 \). Then the function \( \Psi_1(u) \) has the form
\[
\Psi_1(u) := \frac{(\beta u - r)(1+cu)}{u},
\]
which has the derivative
\[
\Psi_1'(u) = \frac{\beta cu^3 + r}{u^2}.
\]
Therefore, \( \Psi_1'(u) > 0 \) when \( u > 0 \), i.e., the function \( \Psi_1(u) \) is monotone increasing.

Moreover,
\[
\Psi_1(1) = (\beta - r)(1+c).
\]
Thus, the equation \eqref{theta} has a solution \( \overline{u} \) in \( (0,1) \) if and only if \( \theta < (\beta - r)(1+c) \). This solution satisfies the equation:
\[
c\beta u^2 + (\beta - rc - \theta)u - r = 0.
\]

Additionally, since \( \overline{u} \) is a solution of \( u^{(1)} = u \), it follows that \( \overline{v} = 1 - \overline{u} \). The proof is complete.

\end{proof}

\begin{pro}\label{prop3} Let $h=1$ and \( q(u) \) be defined as in~\eqref{pq}.
For the positive fixed point \( \overline{E} = (\overline{u}, \overline{v}) \) of the operator~\eqref{h1}, the following holds:
\[
\overline{E} =
\begin{cases}
\text{attractive}, & \text{if } q(\overline{u}) < 1, \\[2mm]
\text{repelling}, & \text{if } q(\overline{u}) > 1, \\[2mm]
\text{nonhyperbolic}, & \text{if } q(\overline{u}) = 1.
\end{cases}
\]
\end{pro}

\begin{proof} Consider the characteristic function~\eqref{chareq} at \( \lambda=1 \):
\[
F(1, \overline{u}) = \overline{u}(1-\overline{u})\left(\beta-\frac{\theta}{(1+c\overline{u})^2}\right).
\]
From this, we obtain that \( F(1, \overline{u})>0 \) is equivalent to
\[
\theta < \beta (1 + c \overline{u})^2.
\]

Define the function
\[
\phi(x) = \beta (1 + c x)^2
\]
and consider the difference \( \phi(x) - \Psi_1(x) \):
\begin{equation}
\phi(x) - \Psi_1(x) = \beta (1 + c x)^2 - \frac{(\beta x - r)(1 + c x)}{x} = \frac{\beta c^2 x^3 + \beta c x^2 + rc x + r}{x}.
\end{equation}
Clearly, \( \phi(x) - \Psi_1(x) > 0 \) for all \( x > 0. \) Therefore, we have \( \theta = \Psi_1(\overline{u}) < \phi(\overline{u}), \) which implies \( F(1, \overline{u}) > 0. \)

If \( F(1, \overline{u}) > 0, \) then \( F(-1, \overline{u}) > 0 \) follows directly. By Lemma~\ref{lem1}, the proof of the proposition is complete.

\end{proof}

\begin{pro}\label{prop4} Let \( v^{(1)} \) be defined as in (\ref{h1}).
If one of the following conditions on the parameters holds, then \( v^{(1)} \geq 0 \) for any \( u \in [0,1] \) and $v\geq0$:

(a) \( 0<r\leq 1 \), \( 0 < \theta \leq 1+\beta-r \) and \( c >0 \);

(b) \( 0<r<1 \), \( 1+\beta-r<\theta\leq\frac{(1+\beta-r)^2}{\beta}\), and \( c\geq\frac{\theta}{1+\beta-r}-1 \);

(c) \( 0<r<1 \), \(\theta>\frac{(1+\beta-r)^2}{\beta}\), and \( c\geq\frac{(\sqrt{\beta}-\sqrt{\theta})^2}{1-r}\).
\end{pro}

\begin{proof}
The condition \( v^{(1)} \geq 0 \) is equivalent to
\[
\beta cu^2+ (\beta + c - r c-\theta) + 1 - r \geq 0,
\]
and the proof of proposition follows directly by solving the inequalities
$u_1\leq0$ and $u_2\geq1,$ where
\[
u_{1,2}=\frac{rc+\theta-\beta-c\mp\sqrt{(rc+\theta-\beta-c)^2-4\beta c(1-r)}}{2\beta c}.
\]
\end{proof}

\begin{lemma}\label{lem2}
Assume that at least one of the conditions \((a)-(c)\) holds and that either \( \beta \leq r \) or \( \theta \geq (\beta - r)(1 + c) \). Then, the following set is invariant under the operator (\ref{h1}):
    \begin{equation}\label{minv}
    M = \left\{ (u,v) \in \mathbb{R}^2 \mid 0 \leq u \leq 1, \, 0 \leq v \leq 2 - u \right\}.
    \end{equation}
\end{lemma}

\begin{proof} By conditions \((a)-(c)\), we have \( v^{(1)} \geq 0 \). Note that if \( 0 \leq u \leq 2 \) and \( 0 < v \leq 2 - u \), then after one iteration, we obtain \( 0 \leq u^{(1)} \leq 1 \).

Additionally, if \( \beta \leq r \) or \( \theta \geq (\beta - r)(1 + c) \), then the sequence \( v^{(n)} \) is decreasing (as established in the proof of Proposition \ref{prop2}). The sequence \( u^{(n)} \) decreases above the line \( v = 1 - u \) and increases below it. We now consider two cases:

Case 1: If \( v \leq 1 - u \), then \( u^{(1)} \geq u \). From \( 0 \leq u^{(1)} \leq 1 \), it follows that
\[
1 \leq 2 - u^{(1)} \leq 2.
\]
Since \( v^{(1)} \leq v \leq 1 \), the inequality \( v^{(1)} > 2 - u^{(1)} \) is impossible. Thus, we conclude that
\[
v^{(1)} \leq 2 - u^{(1)}.
\]

Case 2: If \( v > 1 - u \), then \( u^{(1)} < u \). Since the function \( v = 1 - u \) is decreasing, we obtain
\[
v^{(1)} \leq v \leq 2 - u \leq 2 - u^{(1)}.
\]

Thus, the proof is complete.
\end{proof}

%It is obvious that the set
%\[
% \{(u,v) \in \mathbb{R}_+^2 \mid 0 \leq u \leq 2, v = 0\}
%\]
%is invariant with respect to the operator (\ref{h12}). Moreover, if \( 0 < r \leq 1 \), then the set
%\[
%\{(u,v) \in \mathbb{R}_+^2 \mid u = 0, v \geq 0\}
%\]
%is also invariant under the operator (\ref{h12}).
%
%The dynamics on the coordinate axes Moreover, through straightforward calculations, it can be shown that a trajectory originating from the set \( \mathcal{U} \) converges to the fixed point \( E_1 = (1,0) \), while a trajectory starting from the set \( \mathcal{V} \) converges to the fixed point \( E_0 = (0,0) \).

\begin{pro}\label{prop5}
Assume that at least one of the conditions \((a)-(c)\) holds and that either \( \beta \leq r \) or \( \theta \geq (\beta - r)(1 + c) \). Then, for any initial point \( (u^0,v^0) \in M \), the trajectory of the operator (\ref{h1}) satisfies the following:

- If \( u^0 = 0 \), the trajectory converges to the fixed point \((0,0) \).

- If \( u^0 > 0 \), the trajectory converges to the fixed point \( (1,0) \).

\end{pro}

\begin{proof}
Define the following function on the invariant set \( M \):
\[
L(u,v) = v.
\]
Clearly, \( L(u,v) \geq 0 \) and \( L(0,0) = 0 \), \( L(1,0) = 0 \). Consider the difference:
\[
\Delta L = L(u^{(1)},v^{(1)}) - L(u,v) = v^{(1)} - v \leq 0.
\]
The set where \( \Delta L = 0 \) is given by
\[
\mathcal{B} = \{(u,v) \in M \mid \Delta L = 0\} = \{(u,0)\}.
\]
Since \( M \) is compact, it follows from LaSalle's Invariance Principle that any trajectory starting in \( M \) converges to the fixed point \((0,0) \) if \( u^0 = 0 \), or to the fixed point \((1,0) \) if \( u^0 > 0 \).

\end{proof}

\subsection{Existence of positive fixed point of (\ref{h2})}

We continue our study by investigating the behavior of the function \( \Psi_2(u) \), which is given by
\begin{equation}\label{Psi2}
\Psi_2(u) := \frac{(\beta u - r)(1+cu^2)}{u^2}.
\end{equation}
Its derivative is
\begin{equation}
\Psi_2'(u) = \frac{\beta cu^3 - \beta u + 2r}{u^3}.
\end{equation}
It is evident that one root of the equation
\begin{equation}\label{cub1}
\beta cu^3 - \beta u + 2r = 0
\end{equation}
is always negative. Therefore, the equation (\ref{cub1}) can have at most two positive real roots, denoted by \( \widehat{u}_1 \) and \( \widehat{u}_2 \) with \( \widehat{u}_1 < \widehat{u}_2 \) .

\begin{thm}\label{thm2} Let $\beta> r.$  The operator \eqref{h2} has the following fixed points:

\begin{itemize}
    \item[(i)] Let \( 0 < c \leq \frac{1}{3} \). The following statements hold:
    \begin{itemize}
        \item[(i.1)] If \( c \geq \frac{\beta - 2r}{\beta} \) and \( 0 < \theta < (\beta - r)(1+c) \), then there exists a \textbf{unique} positive fixed point.
        \item[(i.2)] If \( c < \frac{\beta - 2r}{\beta} \) and \( \theta = \Psi_2(\widehat{u}_1) \), then there exists a \textbf{unique} positive fixed point.
        \item[(i.3)] If \( c < \frac{\beta - 2r}{\beta} \) and \( (\beta - r)(1+c) < \theta < \Psi_2(\widehat{u}_1) \), then there exist \textbf{two} positive fixed points.
    \end{itemize}

    \item[(ii)] Let \( c > \frac{1}{3} \). The following statements hold:
    \begin{itemize}
        \item[(ii.1)] If \( c \geq \frac{\beta^2}{27r^2} \) and \( 0 < \theta < (\beta - r)(1+c) \), then there exists a \textbf{unique } positive fixed point.
        \item[(ii.2)] If \( c \leq \frac{\beta - 2r}{\beta} \) and \( \theta = \Psi_2(\widehat{u}_1) \), then there exists a \textbf{unique} positive fixed point.
        \item[(ii.3)] If \( c < \frac{\beta - 2r}{\beta} \) and \( (\beta - r)(1+c) < \theta < \Psi_2(\widehat{u}_1) \), then there exist \textbf{two} positive fixed points.
        \item[(ii.4)] If \( \frac{\beta - 2r}{\beta} < c < \frac{\beta^2}{27r^2} \) and \( \theta = \Psi_2(\widehat{u}_1) \) or \( \theta = \Psi_2(\widehat{u}_2) \), then there exists a \textbf{unique} positive fixed point.
        \item[(ii.5)] If \( \frac{\beta - 2r}{\beta} < c < \frac{\beta^2}{27r^2} \) and \( \Psi_2(\widehat{u}_2) < \theta < \Psi_2(\widehat{u}_1) \), then there exist \textbf{three} positive fixed points.
    \end{itemize}
\end{itemize}
\end{thm}

\begin{proof}
Recall that the operator (\ref{h12}) does not have a positive fixed point when $\beta \leq r$. Hence, we assume that $\beta > r$. Let us investigate the equation (\ref{cub1}), as it characterizes the function $\Psi_2(u)$. Define the function
\begin{equation}\label{fun}
    f(u) = \beta c u^3 - \beta u + 2r.
\end{equation}

It can be shown that the function $f(u)$ has critical points at $\pm\frac{1}{\sqrt{3c}}$. We consider only the positive critical point, denoted as $\widetilde{u} = \frac{1}{\sqrt{3c}}$. At $\widetilde{u}$, the function $f(u)$ attains a local minimum, implying that $f(u)$ decreases for $u < \widetilde{u}$ and increases for $u > \widetilde{u}$. We now consider two cases: $\widetilde{u} \geq 1$ and $\widetilde{u} < 1$.

\textbf{Case (i):} Suppose $\widetilde{u} \geq 1$. Then, we have $0 < c \leq \frac{1}{3}$. It is evident that $f(0) = 2r > 0$.

\textbf{Subcase (i.1):} If $f(1) = \beta c - \beta + 2r \geq 0$, then equation (\ref{cub1}) has no solution in $(0,1)$ (Fig. \ref{fig2}, (a)). Consequently, $\Psi'_2(u)$ is positive in $(0,1)$, implying that $\Psi_2(u)$ is monotonically increasing in $(0,1)$. Thus, if $0 < \theta < (\beta - r)(1+c)$, the curves $y = \theta$ and $y = \Psi_2(u)$ intersect at a unique point in $(0,1)$, meaning that the operator (\ref{h2}) has a unique positive fixed point.

\textbf{Subcase (i.2):} If $f(1) = \beta c - \beta + 2r < 0$, then equation (\ref{cub1}) has a unique solution $\widehat{u}_1$ in $(0,1)$ (Fig. \ref{fig2}, (b)). This implies that $\Psi_2(u)$ is monotonically increasing in $(0,\widehat{u}_1)$ and monotonically decreasing in $(\widehat{u}_1, 1)$, so $\Psi_2(u)$ has a local maximum at $\widehat{u}_1$. Therefore, if $\theta = \Psi_2(\widehat{u}_1)$, the operator (\ref{h2}) has a unique positive fixed point.

\textbf{Subcase (i.3):} As a continuation of subcase (i.2), if $(\beta - r)(1+c) < \theta < \Psi_2(\widehat{u}_1)$, then the curves $y = \theta$ and $y = \Psi_2(u)$ intersect at two points in $(0,1)$. Consequently, the operator (\ref{h2}) has two positive fixed points.

\textbf{Case (ii):} Suppose $\widetilde{u} < 1$. Then, we have $c >\frac{1}{3}$. Now, the number of solutions of equation (\ref{cub1}) depends not only on the sign of $f(1)$ but also on the sign of $f(\widetilde{u})$.

\textbf{Subcase (ii.1):} Let $f(\widetilde{u})\geq0$, i.e., $c \geq \frac{\beta^2}{27r^2}$. Then $f(1)>0$ must be satisfied, and we have that $f(u)\geq0$ in $(0,1)$ (Fig. \ref{fig2}, (c)), which ensures that the function $\Psi_2(u)$ is monotonically increasing in $(0,1)$. Thus, the operator (\ref{h2}) has a unique positive fixed point when $0 < \theta < (\beta - r)(1+c)$.

\textbf{Subcase (ii.2):} Let $f(1)\leq0$. Then the equation $f(u)=0$ has a unique positive solution $\widehat{u}_1$ in $(0,1)$ (Fig. \ref{fig2}, (d)), where the function $\Psi_2(u)$ attains a maximum value. Thus, the operator (\ref{h2}) has a unique positive fixed point when $\theta = \Psi_2(\widehat{u}_1)$.

\textbf{Subcase (ii.3):} As a continuation of subcase (ii.2), if $(\beta - r)(1+c) < \theta < \Psi_2(\widehat{u}_1)$, then the operator (\ref{h2}) has two positive fixed points.

\textbf{Subcase (ii.4):} Let $f(1)>0$ and $f(\widetilde{u})<0$, i.e., $\frac{\beta - 2r}{\beta} < c < \frac{\beta^2}{27r^2}$. Then the equation (\ref{cub1}) has two solutions $\widehat{u}_1$ and $\widehat{u}_2$ ($\widehat{u}_1<\widehat{u}_2$) in $(0,1)$ (Fig. \ref{fig2}, (e)), where $\Psi_2(\widehat{u}_1)$ is a local maximum while $\Psi_2(\widehat{u}_2)$ is a local minimum. Hence, the operator (\ref{h2}) has a unique positive fixed point when $\theta = \Psi_2(\widehat{u}_1)$ or $\theta = \Psi_2(\widehat{u}_2)$.

\textbf{Subcase (ii.5):} As a continuation of subcase (ii.4), if $\Psi_2(\widehat{u}_2) < \theta < \Psi_2(\widehat{u}_1)$, then the curves $y = \theta$ and $y = \Psi_2(u)$ intersect at three points in $(0,1)$, implying that the operator (\ref{h2}) has three positive fixed points. The proof is complete.
\end{proof}

\begin{figure}[h!]
    \centering
    \subfigure[\tiny$c=0.25, \beta=1, r=0.5$]{\includegraphics[width=0.3\textwidth]{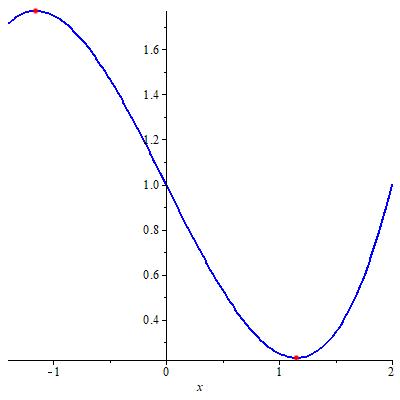}} \hspace{0.3in}
    \subfigure[\tiny$c=0.25, \beta=1, r=0.2$]{\includegraphics[width=0.3\textwidth]{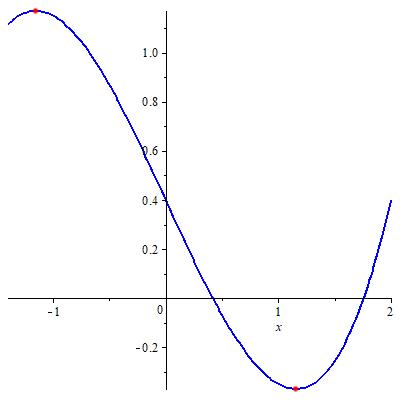}}
    \subfigure[\tiny$c=0.5, \beta=1, r=0.5$]{\includegraphics[width=0.3\textwidth]{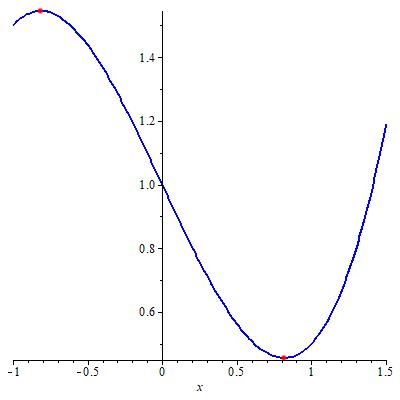}} \hspace{0.3in}
    \subfigure[\tiny$c=0.5, \beta=1, r=0.1$]{\includegraphics[width=0.3\textwidth]{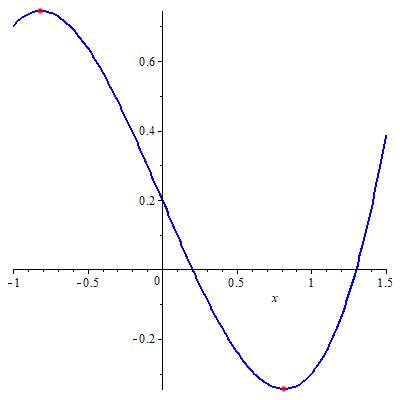}}
    \subfigure[\tiny$c=0.51, \beta=2, r=0.5$]{\includegraphics[width=0.3\textwidth]{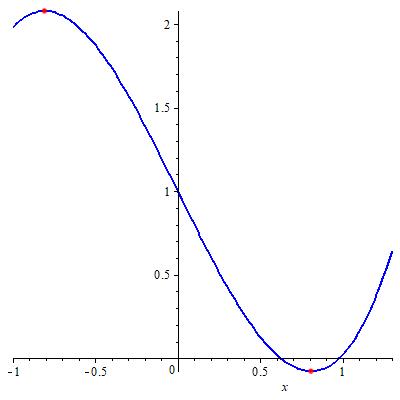}}\hspace{0.3in}
    \subfigure[\tiny$c=1, \beta=3, r=0.5$]{\includegraphics[width=0.3\textwidth]{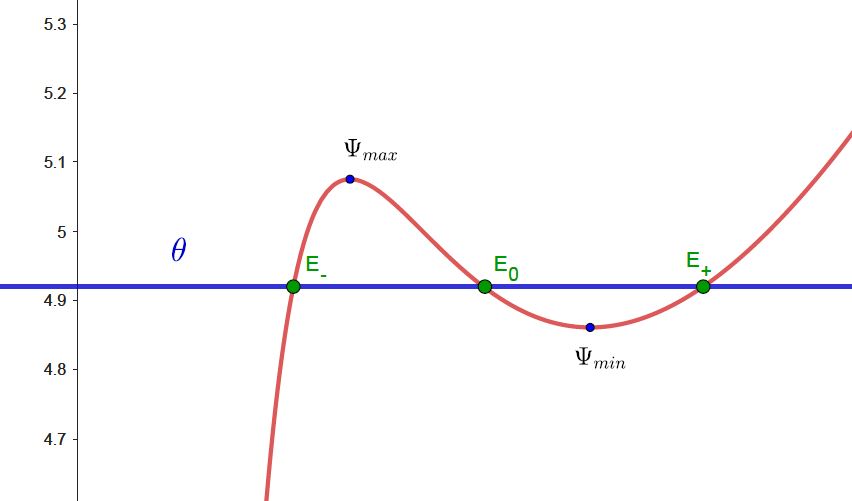}}
    \caption{Graph of \( f(u) = \beta c u^3 - \beta u + 2r \) for various parameter values of \( \beta \), \( c \), and \( r \).}
    \label{fig2}
\end{figure}

\subsection{Classification of fixed points of (\ref{h2}).}

Let $E=(u_{-}, v_{-})$, $E_0=(u_{0}, v_{0})$, and $E_{+}=(u_{+}, v_{+})$ be positive fixed points of the operator \eqref{h2} such that $u_{-} < u_{0} < u_{+}$.

\begin{thm}\label{type}  Let $h=2$ and \( q(u) \) be defined as in~\eqref{pq}.
For the fixed points $E_{\mp} = (u_{\mp}, v_{\mp})$ and $E_0 = (u_{0}, v_{0})$ of the operator \eqref{h2}, the following hold true:

\begin{itemize}
    \item[(i)] The nature of $E_{-}$ is determined as follows:
    \[
    E_{-} =
    \begin{cases}
    \text{attractive}, & \text{if } q(u_{-}) < 1, \\
    \text{repelling}, & \text{if } q(u_{-}) > 1, \\
    \text{nonhyperbolic}, & \text{if } q(u_{-}) = 1.
    \end{cases}
    \]

    \item[(ii)] The fixed point $E_{0}$ is saddle if it exists.

    \item[(iii)] The fixed point $E_{+}$ is attracting if it exits.
    \end{itemize}
\end{thm}

\begin{proof}
Recall that $\widehat{u}_1$ and $\widehat{u}_2$ are the critical points of the function $\Psi_2(u)$, i.e., $\Psi'_2(\widehat{u}_1) = \Psi'_2(\widehat{u}_2) = 0$, where $\Psi_2(u)$ is defined in \eqref{Psi2}. Moreover, $f(u) > 0$ for $u \in (0,\widehat{u}_1) \cup (\widehat{u}_2,1)$ and $f(u) < 0$ for $u \in (\widehat{u}_1,\widehat{u}_2)$, where $f(u)$ is defined in \eqref{fun}. It is evident that $u_{-} \in (0,\widehat{u}_1)$, $u_{0} \in (\widehat{u}_1,\widehat{u}_2 )$, and $u_{+} \in (\widehat{u}_2,1 )$ (Fig \ref{fig2}, (f)).

\begin{itemize}
    \item[(i)] We show that $F(1, u_{\mp}) > 0$. From \eqref{chareq}, this is equivalent to
    \[
    \theta < \frac{\beta(1 + c u_{\mp}^2)^2}{2u_{\mp}}.
    \]
    Define the function
    \[
    \varphi(x) = \frac{\beta(1 + c x^2)^2}{2x},
    \]
    and consider the difference $\varphi(x) - \Psi_2(x)$:
    \begin{equation}
    \varphi(x) - \Psi_2(x) = \frac{c x^2 (\beta c x^3 + 2r) - \beta x + 2r}{2x^2}.
    \end{equation}
    Recall that, if $x \in (0,\widehat{u}_1) \cup (\widehat{u}_2,1)$ then $f(x) > 0$, i.e., $\beta c x^3 + 2r > \beta x$, it follows that
    \[
    \varphi(x) - \Psi_2(x) > \frac{\beta c x^3 - \beta x + 2r}{2x^2} > 0.
    \]
    Thus, we conclude that $\theta = \Psi_2(u_{\mp}) < \varphi(u_{\mp})$, which gives $\theta < \frac{\beta(1 + c u_{\mp}^2)^2}{2u_{\mp}}$. Since $F(1, u_{\mp}) > 0$, it follows that $F(-1, u_{\mp}) > 0$. By Lemma \ref{lem1}, this case is proved.

    \item[(ii)]
Similarly, if $x \in (\widehat{u}_1, \widehat{u}_2)$, then $f(x) < 0$, i.e.,
\begin{equation}
    \beta c x^3 + 2r < \beta x,
\end{equation}
which implies that $\varphi(x) - \Psi_2(x) < 0$. Consequently, we obtain that $F(1, u_{0}) < 0$, meaning one eigenvalue of the Jacobian at $(u_0,v_0)$ lies in $(1,\infty)$.

Now, we will show that $F(-1, u_{0}) > 0$ taking into account that $ \theta = \Psi_2(u_{0})$:
\begin{equation}
    F(-1, u_{0})=4-2u_0+\frac{(1 - u_{0})(\beta c u_{0}^3 - \beta u_{0} + 2r)}{1 + c u_{0}^2}>0.
\end{equation}
Note that $\beta c u_{0}^3 - \beta u_{0} + 2r<0$. Then we get
\begin{equation}
   F(-1, u_{0})>0 \ \ \Leftrightarrow  \ \ \beta u_{0}-\beta c u_{0}^3-2r< \frac{(4-2u_0)(1 + c u_{0}^2)}{1 - u_{0}}.
\end{equation}
Consider the function
\begin{equation}
    \nu(x) = \frac{(4-2x)(1 + c x^2)}{1 - x}.
\end{equation}
Obviously, $\nu(0) = 4$ and $\nu(x) \to +\infty$ as $x \to 1$. Moreover, for all $x \in (0,1)$, we observe that $\nu(x) > 0$. Consider $\nu'(x):$
\begin{equation}
    \nu'(x)=\frac{4cx^3-10cx^2+8cx+2}{(1-x)^2}.
\end{equation}
The numerator of $\nu'(x)$ is the function $4cx^3-10cx^2+8cx+2$, which has a minimum value of $\frac{56c}{27}+2>0$ at $x=\frac{2}{3}$. Thus, we obtain that $\nu'(x)>0$, which means $\nu(x)$ is monotonically increasing in $(0,1)$.

In addition, the function $\beta x-\beta c x^3-2r$ has a maximum value of $\frac{2\beta}{\sqrt{3c}}-2r$. According to Theorem \ref{thm2}, for the existence of more than one positive fixed point, it is necessary that $c<\frac{\beta^2}{27r^2}$. From this condition, we conclude that the maximum value $\frac{2\beta}{\sqrt{3c}}-2r$ is less than 4, which means the functions $\nu(x)$ and $\beta x-\beta c x^3-2r$ do not intersect in $(0,1)$ (see Fig. \ref{figg1}, (a)). Hence, we have proved that $F(-1, u_{0}) > 0$, which implies that the fixed point $E_0$ is a saddle if it exists.

    \item[(iii)] In the previous discussion, we established that \( F(1, u_{+}) > 0 \) and \( F(-1, u_{+}) > 0 \). Now, consider the condition \( q(u_{+}) < 1 \), taking into account that \( \theta = \Psi_2(u_{+}) \):

\[
    q(u_{+}) < 1 \Rightarrow  \frac{(1 - u_{+})(\beta c u_{+}^3 - \beta u_{+} + 2r)}{u_{+} (1 + c u_{+}^2)} < 1 \Rightarrow
    \beta c u_{+}^3 - \beta u_{+} + 2r < \frac{u_{+} + c u_{+}^3}{1 - u_{+}}.
\]

Next, we define the function:

\[
    \omega(x) = \frac{x + c x^3}{1 - x}.
\]

It is evident that \( \omega(0) = 0 \) and \( \omega(x) \to +\infty \) as \( x \to 1 \). Moreover, for all \( x \in (0,1) \), we observe that:
\begin{itemize}
    \item \( \omega(x) > 0 \),
    \item \( \omega'(x) > 0 \),
    \item \( \omega''(x) > 0 \).
\end{itemize}
These properties indicate that \( \omega(x) \) is monotonically increasing in \( (0,1) \) and its graph is concave.

Now, consider the behavior of the function \( f(x) = \beta c x^3 - \beta x + 2r \). It is evident that \( f(x) \) is also concave in \( (0,1) \). Consequently, the functions \( f(x) \) and \( \omega(x) \) intersect at a unique point \( \widetilde{x} \) in the first quadrant, where \( \widetilde{x} < \widehat{u}_1 \). Since \( u_{+} > \widehat{u}_{2} > \widehat{u}_{1} \), it follows that \( f(u_{+}) < \omega(u_{+}) \), proving that \( \overline{q}(u_{+}) < 1 \) (see Fig. \ref{figg1}, (b)).

Finally, by Lemma \ref{lem1}, the proof of the theorem is complete.

        \begin{figure}[h!]
    \centering
    \subfigure[\tiny The function \( \nu(x) \) is represented in blue, while the negative function \( -f(x) \) is shown in red.]{\includegraphics[width=0.35\textwidth]{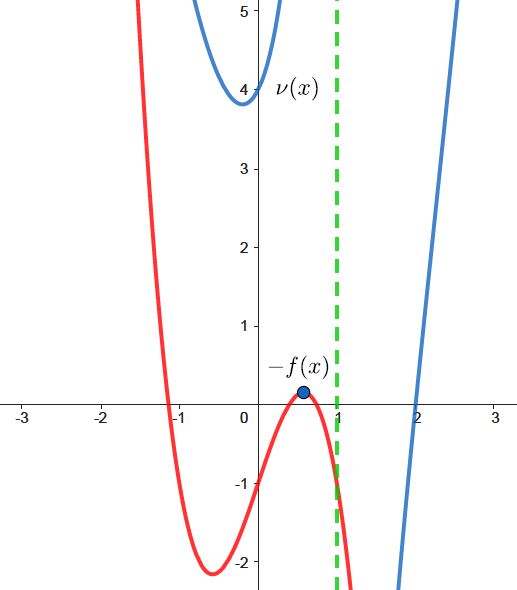}} \hspace{0.3in}
    \subfigure[\tiny The function \( \omega(x) \) is represented in blue, while \( f(x) \) is shown in red.]{\includegraphics[width=0.35\textwidth]{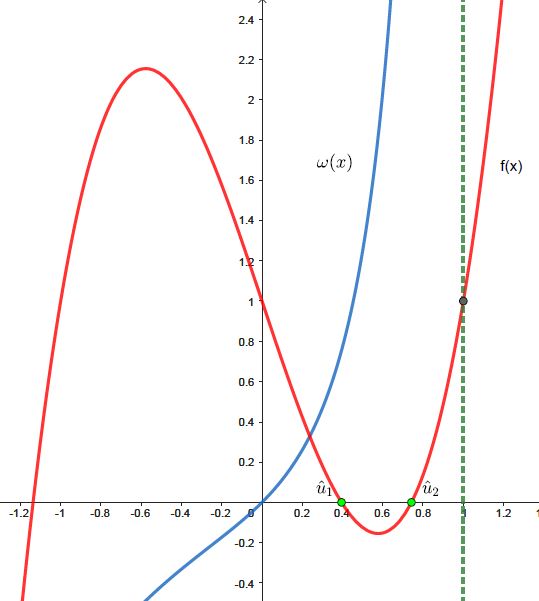}}
    \caption{To prove Theorem \ref{type}, we introduced new functions $\nu(x)$ and $\omega(x).$}
    \label{figg1}
\end{figure}
\end{itemize}
\end{proof}

Now we will find the conditions on the parameters for the set $M$ to be an invariant set with respect to operator (\ref{h2}). Define the function

\begin{equation}\label{poly}
\psi(u)=\beta c u^3-(rc-c+\theta)u^2+\beta u+1-r
\end{equation}

\begin{pro}
Let \( v^{(1)} \) be defined as in (\ref{h2}). If
\begin{equation}\label{con1}
0 < r \leq 1, \quad 0 < \theta \leq \min\left\{(1+c)(\beta-r+1),\ (3+c)(1-r)+2\beta\right\},
\end{equation}
then \( v^{(1)} \geq 0 \) for any \( u \in [0,1] \) and \( v \geq 0 \).
\end{pro}

\begin{proof}
We express the polynomial \( \psi(u) \) in the Bernstein basis polynomial form:
\[
\psi(u) = \sum_{i=0}^{3} \omega_i \binom{3}{i} u^i (1 - u)^{3 - i}.
\]
From this representation, we obtain the following system of equations:
\[
\left\{
\begin{aligned}
&\omega_0 = 1 - r, \\
&-3\omega_0 + 3\omega_1 = \beta, \\
&3\omega_0 - 6\omega_1 + 3\omega_2 = -(rc - c + \theta), \\
&-\omega_0 + 3\omega_1 - 3\omega_2 + \omega_3 = \beta c.
\end{aligned}
\right.
\]
Solving this linear system yields:
\[
\left\{
\begin{aligned}
&\omega_0 = 1 - r, \\
&\omega_1 = \frac{\beta}{3} + 1 - r, \\
&\omega_2 = \frac{2\beta - rc + c - \theta}{3} + 1 - r, \\
&\omega_3 = (c + 1)(1 - r + \beta) - \theta.
\end{aligned}
\right.
\]
It is evident that if \( \omega_i \geq 0 \) for all \( i = 0, 1, 2, 3 \), then \( \psi(u) \geq 0 \) for all \( u \in [0,1] \). The conditions stated in the proposition ensure the non-negativity of all coefficients \( \omega_i \). Thus, the proof is complete.
\end{proof}

\begin{lemma}
Assume that the conditions
 \[
0 < r \leq 1, \quad 0 < \theta \leq \min\left\{(1+c)(\beta-r+1),\ (3+c)(1-r)+2\beta\right\},
\]
hold and that either \( \beta \leq r \) or \( \theta \geq (\beta - r)(1 + c) \). Then, the set $M$ is invariant w.r.t operator (\ref{h2}), where $M$ is defined as in (\ref{minv}).
\end{lemma}
The proof is similar to the proof of Lemma \ref{lem2} .

\begin{pro}\label{prop7}
Assume that the conditions
  \[
0 < r \leq 1, \quad 0 < \theta \leq \min\left\{(1+c)(\beta-r+1),\ (3+c)(1-r)+2\beta\right\},
\]
hold and that either \( \beta \leq r \) or \( \theta \geq (\beta - r)(1 + c) \). Then, for any initial point \( (u^0,v^0) \in M \), the trajectory of the operator (\ref{h2}) satisfies the following:

- If \( u^0 = 0 \), the trajectory converges to the fixed point \( E_0 = (0,0) \).

- If \( u^0 > 0 \), the trajectory converges to the fixed point \( E_1 = (1,0) \).
\end{pro}
The proof is similar to the proof of Proposition \ref{prop5} .

\section{Bifurcation analysis of the model (\ref{h12})}

Let $\widetilde{E}=(\widetilde{u}, \widetilde{v})$ be a unique positive fixed point $\overline{E}=(\overline{u}, \overline{v})$ of the model (\ref{h1}) or positive fixed point $E_{-}=(u_{-},v_{-})$ of the model (\ref{h2}).

In this section, we derive the conditions under which a Neimark-Sacker bifurcation occurs at the positive fixed point $\widetilde{E}=(\widetilde{u}, \widetilde{v})$.

Assume that the parameters satisfy the existence conditions for $\widetilde{E}=(\widetilde{u}, \widetilde{v})$ , i.e., $\beta>r$ and $\theta<(\beta-r)(1+c)$. Let $q(u)$ be defined as in (\ref{pq}) and \( \theta_0 \) denote a solution of \( q(\widetilde{u}) = 1 \). Thus, the fixed point \( \widetilde{E} \) can undergo a Neimark-Sacker bifurcation as the parameter \( \theta \) varies within a small neighborhood of \( \theta_0 \). To clearly illustrate this process, we proceed as follows.

\textbf{Step-1}. We introduce the change of variables \( x = u - \widetilde{u} \), \( y = v - \widetilde{v} \), which shifts the fixed point \( \widetilde{E} = (\widetilde{u}, \widetilde{v}) \) to the origin. Under this transformation, system~\eqref{h12} becomes

\begin{equation}\label{bif1}
\left\{\begin{aligned}
&x^{(1)}=(x+\widetilde{u}) (1-x-y)-\widetilde{u}\\
&y^{(1)}=(y+\widetilde{v}) \left(\beta(x+\widetilde{u})+1-r-\frac{\theta(x+\widetilde{u})^h}{1+c(x+\widetilde{u})^h}\right)-\widetilde{v}.
\end{aligned}\right.
\end{equation}

\textbf{Step-2}. We give a small perturbation \( \theta^* \) to the parameter \( \theta \), such that \(
\theta = \theta_0 + \theta^*,\)
where \( \theta_0 \) denotes the bifurcation value of \( \theta \), and \( \theta^* \) is a small real parameter. Substituting this into system~\eqref{bif1}, the perturbed system takes the form:

\begin{equation}\label{bif2}
\left\{\begin{aligned}
&x^{(1)}=(x+\widetilde{u}) (1-x-y)-\widetilde{u}\\
&y^{(1)}=(y+\widetilde{v}) \left(\beta(x+\widetilde{u})+1-r-\frac{(\theta_0+\theta^*)(x+\widetilde{u})^h}{1+c(x+\widetilde{u})^h}\right)-\widetilde{v}.
\end{aligned}\right.
\end{equation}

Since \( \widetilde{u} \) is a solution of the equation \( \theta = \Psi_h(u) \), we can simplify the Jacobian matrix of the perturbed system~\eqref{bif2} evaluated at the origin \( (0,0) \). This yields:
\begin{equation}\label{jac2}
J(0,0)=\begin{bmatrix}
1-\widetilde{u} &~~~ -\widetilde{u}\\
(1-\widetilde{u})\left(\beta-\frac{h(\theta_0+\theta^*)\widetilde{u}^{h-1}}{(1+c\widetilde{u}^h)^2}\right) &~~~ 1-\frac{\theta^*\widetilde{u}^h}{1+c\widetilde{u}^h}
\end{bmatrix}
\end{equation}
and its characteristic equation is
$$\lambda^2-a(\theta^*)\lambda+b(\theta^*)=0,$$
where
$$a(\theta^*)=Tr(J)=2-\widetilde{u}-\frac{\theta^*\widetilde{u}^h}{1+c\widetilde{u}^h}=p(\widetilde{u})-\frac{\theta^*\widetilde{u}^h}{1+c\widetilde{u}^h},$$
 and since $\theta_0$ is a solution of $q(\widetilde{u})=1,$ we get
\begin{align*}
 b(\theta^*)&=\det(J)=1-\frac{\theta^*\widetilde{u}^h(1-\widetilde{u})(1+h+c\widetilde{u}^h)}{(1+c\widetilde{u}^h)^2}.
\end{align*}
Note that \( a(0) = p(\widetilde{u}) < 2 \) and \( b(0) = 1 \), hence we obtain the following pair of complex conjugate roots:

\begin{equation}\label{bif5}
\lambda_{1,2}=\frac{1}{2}[a(\theta^*)\pm i \sqrt{4b(\theta^*)-a^2(\theta^*)}].
\end{equation}
Thus, \(|\lambda_{1,2}|=\sqrt{b(\theta^*)}\) and

\begin{equation}\label{bif7}
\frac{d|\lambda_{1,2}|}{d\theta^*}\Bigm|_{\theta^*=0}\!=\!-\frac{\widetilde{u}^h(1-\widetilde{u})(1+h+c\widetilde{u}^h)}{2\sqrt{b(\theta^*)}(1+c\widetilde{u}^h)^2}\Bigm|_{\theta^*=0}\!=\!-\frac{\widetilde{u}^h(1-\widetilde{u})(1+h+c\widetilde{u}^h)}{2(1+c\widetilde{u}^h)^2}<0.
\end{equation}

Therefore, the transversality condition
\[
\frac{d|\lambda_{1,2}|}{d\theta^*} \Big|_{\theta^* = 0} \neq 0
\]
has been established. Next, we consider the non-degeneracy condition, i.e., \( \lambda_{1,2}^m(0) \neq 1 \) for \( m = 1, 2, 3, 4 \). Given that \( 1 < a(0) < 2 \) and \( b(0) = 1 \), it follows directly that \( \lambda_{1,2}^m(0) \neq 1 \) for all \( m = 1, 2, 3, 4 \). Hence, all the conditions for the occurrence of a Neimark–Sacker bifurcation are satisfied.

\textbf{Step-3}. To derive the normal form of the system (\ref{bif2}) when \( \theta^* = 0 \), we expand the system (\ref{bif2}) into a Taylor series around \( (x, y) = (0, 0) \) up to third order.
\begin{equation}\label{bif8}
\left\{\begin{aligned}
&x^{(1)}=a_{10}x\!+\!a_{01}y\!+\!a_{20}x^2\!+\!a_{11}xy\!+\!a_{02}y^2\!+\!a_{30}x^3\!+\!a_{21}x^2y\!+\!a_{12}xy^2\!+\!a_{03}y^3\!+\!O(\rho^4)\\
&y^{(1)}= b_{10}x\!+\!b_{01}y\!+\!b_{20}x^2\!+\!b_{11}xy\!+\!b_{02}y^2\!+\!b_{30}x^3\!+\!b_{21}x^2y\!+\!b_{12}xy^2\!+\!b_{03}y^3\!+\!O(\rho^4),
\end{aligned}\right.
\end{equation}
where $\rho=\sqrt{x^2+y^2}$  and
\begin{equation}
\begin{split}
&a_{10}=1-\widetilde{u}, \ \ a_{01}=-\widetilde{u}, \ \ a_{20}=-1, \ \ a_{11}=-1,\\
&a_{02}=a_{03}=a_{12}=a_{30}=a_{21}=0, \\
&b_{10}=(1-\widetilde{u})\left(\beta-\frac{h\theta_0\widetilde{u}^{h-1}}{(1+c\widetilde{u}^h)^2}\right)=1  (since \ \ q(\widetilde{u})=1),\\
&b_{01}=1, \ \ b_{02}=b_{03}=b_{12}=0, \\
&b_{20}=\frac{h\theta_0(1-\widetilde{u})\widetilde{u}^{h-2}(1+c\widetilde{u}^h+h(c\widetilde{u}^h-1))}{2(1+c\widetilde{u}^h)^3}, \\ &b_{11}=\beta-\frac{h\theta_0\widetilde{u}^{h-1}}{(1+c\widetilde{u}^{h})^2},\ \
b_{21}=\frac{h\theta_0\widetilde{u}^{h-2}(1+c\widetilde{u}^h+h(c\widetilde{u}^h-1))}{2(1+c\widetilde{u}^h)^3}, \\ &b_{30}=-\frac{h\theta_0(1-\widetilde{u})\widetilde{u}^{h-3}\left(2(1+c\widetilde{u}^{h})^2+3h(c^2\widetilde{u}^{2h}-1)+h^2(1-4c\widetilde{u}^{h}+c^2\widetilde{u}^{2h})\right)}{6(1+c\widetilde{u}^{h})^4}.
\end{split}
\end{equation}
Then
 \[
 J(\widetilde{E})=\begin{bmatrix}
a_{10} & ~~a_{01}\\
b_{10} &~~ b_{01}
\end{bmatrix}
\]
and eigenvalues of the Jacobian are:
$$\lambda_{1,2}=\frac{2-\widetilde{u}\mp i\sqrt{4\widetilde{u}-\widetilde{u}^2}}{2}.$$

Corresponding eigenvectors are
$$v_{1,2}=\begin{bmatrix}-\frac{\widetilde{u}}{2}\\1\end{bmatrix}\mp i\begin{bmatrix} \frac{\sqrt{4\widetilde{u}-\widetilde{u}^2}}{2} \\ 0\end{bmatrix}.$$

\textbf{Step-4}. Now we find the normal form of the system (\ref{bif2}). We rewrite the system (\ref{bif8}) in the following form:
\begin{equation}\label{sf}
\mathbf{x}^{(1)}=J\cdot\mathbf{x}+h(\mathbf{x})
\end{equation}
where $\mathbf{x}=(x,y)^T$ and $h(\mathbf{x})$ is nonlinear part of the system (\ref{bif8}) without $O$ notion, i.e.,
\[
h(\mathbf{x})=\begin{bmatrix} -x^2-xy\\ b_{20}x^2+b_{11}xy+b_{30}x^3+b_{21}x^2y\end{bmatrix}
\]

Let matrix
\[
T= \begin{bmatrix} \frac{\sqrt{4\widetilde{u}-\widetilde{u}^2}}{2} & -\frac{\widetilde{u}}{2}\\ 0 & 1\end{bmatrix}
\]
then
 \[
 T^{-1}= \begin{bmatrix} \frac{2}{\sqrt{4\widetilde{u}-\widetilde{u}^2}} & \frac{\widetilde{u}}{\sqrt{4\widetilde{u}-\widetilde{u}^2}} \\
 0 & 1\end{bmatrix}.
 \]

By transformation, we get that
\[
\begin{bmatrix} x \\ y\end{bmatrix}=T\cdot \begin{bmatrix} X \\ Y\end{bmatrix}.
\]

the system (\ref{bif8}) transforms into the following system

\begin{equation}\label{sf1}
\mathbf{X}^{(1)}=T^{-1}\cdot J\cdot T\cdot\mathbf{X}+T^{-1}\cdot h(T\cdot\mathbf{x})+O(\rho_1)
\end{equation}
where $\mathbf{X}=(X,Y)^T$ and $\rho_1=\sqrt{X^2+Y^2}.$ Denote
\[
h(T\cdot\mathbf{x})=\begin{bmatrix} f(X,Y) \\ g(X,Y)\end{bmatrix}
\]
where $s=\sqrt{4\widetilde{u}-\widetilde{u}^2}$ and
\[
\begin{aligned}
f(X,Y)=&-\frac{s^2}{4}X^2+\frac{s(\widetilde{u}-1)}{2}XY+\frac{\widetilde{u}(2-\widetilde{u})}{4}Y^2,\\
g(X,Y)=&\frac{b_{20}s^2}{4}X^2+\frac{s(b_{11}-b_{20}\widetilde{u})}{2}XY+\frac{\widetilde{u}(b_{20}\widetilde{u}-2b_{11})}{4}Y^2+\frac{b_{30}s^3}{8}X^3+\\ & +\frac{s^2(2b_{21}-3b_{30}\widetilde{u})}{8}X^2Y+\frac{s\widetilde{u}(3b_{30}\widetilde{u}-4b_{21})}{8}XY^2+\frac{\widetilde{u}^2(2b_{21}-b_{30}\widetilde{u})}{8}Y^3.
\end{aligned}
\]
If we denote
\[
T^{-1}\cdot h(T\cdot\mathbf{x})=\begin{bmatrix} F(X,Y) \\ G(X,Y)\end{bmatrix}
\]
then we have
\begin{equation}
\begin{split}
&F(X,Y)=c_{20}X^2+c_{11}XY+c_{02}Y^2+c_{30}X^3+c_{21}X^2Y+c_{12}XY^2+c_{03}Y^3,\\
&G(X,Y)=d_{20}X^2+d_{11}XY+d_{02}Y^2+d_{30}X^3+d_{21}X^2Y+d_{12}XY^2+d_{03}Y^3.
\end{split}
\end{equation}

where
\begin{equation}
\begin{split}
&c_{20}=\frac{s(b_{20}\widetilde{u}-2)}{4}, \ \ c_{11}=\frac{2\widetilde{u}-2+\widetilde{u}(b_{11}-b_{20}\widetilde{u})}{2},\\
&c_{02}=\frac{2\widetilde{u}(2-\widetilde{u})+\widetilde{u}^2(b_{20}\widetilde{u}-2b_{21})}{4s}, \ \ c_{30}= \frac{s^2b_{30}\widetilde{u}}{8},\\
&c_{21}=\frac{s\widetilde{u}(2b_{21}-3b_{30}\widetilde{u})}{8}, \ \ c_{12}= \frac{\widetilde{u}^2(3b_{30}\widetilde{u}-4b_{21})}{8},\ \ c_{03}= \frac{\widetilde{u}^3(2b_{21}-b_{30}\widetilde{u})}{8s},\\
&d_{20}=\frac{b_{20}s^2}{4},\ \ d_{11}=\frac{s(b_{11}-b_{20}\widetilde{u})}{2},\ \ d_{02}=\frac{\widetilde{u}(b_{20}\widetilde{u}-2b_{11})}{4},\ \ d_{30}=\frac{b_{30}s^3}{8},\\
&d_{21}=\frac{s^2(2b_{21}-3b_{30}\widetilde{u})}{8},\ \ d_{12}=\frac{s\widetilde{u}(3b_{30}\widetilde{u}-4b_{21})}{8},\ \ d_{03}=\frac{\widetilde{u}^2(2b_{21}-b_{30}\widetilde{u})}{8}.\\
\end{split}
\end{equation}

In addition, the partial derivatives at $(0,0)$ are
\begin{equation}
\begin{split}
&F_{XX}=2c_{20}, \ \ F_{XY}=c_{11}, \ \ F_{YY}=2c_{02}, \\
&F_{XXX}=6c_{30}, \ \ F_{XXY}=2c_{21}, \ \ F_{XYY}=2c_{12}, \ \ F_{YYY}=6c_{03}\\
&G_{XX}=2d_{20}, \ \ G_{XY}=d_{11}, \ \ G_{YY}=2d_{02}, \\
&G_{XXX}=6d_{30}, \ \ G_{XXY}=2d_{21}, \ \ G_{XYY}=2d_{12}, \ \ G_{YYY}=6d_{03}.
\end{split}
\end{equation}

\textbf{Step-5}. We need to compute the discriminating quantity \( \mathcal{L} \) using the following formula, which determines the stability of the invariant closed curve bifurcated from the Neimark-Sacker bifurcation of the system \eqref{sf1}:
\begin{equation}\label{lya}
\mathcal{L}=-Re\left[\frac{(1-2\lambda_1)\lambda_2^2}{1-\lambda_1}L_{11}L_{20}\right]-\frac{1}{2}|L_{11}|^2-|L_{02}|^2+Re(\lambda_2 L_{21}),
\end{equation}
where
\begin{equation}
\begin{split}
&L_{20}=\frac{1}{8}[(F_{XX}-F_{YY}+2G_{XY})+i(G_{XX}-G_{YY}-2F_{XY})],\\
&L_{11}=\frac{1}{4}[(F_{XX}+F_{YY})+i(G_{XX}+G_{YY})],\\
&L_{02}=\frac{1}{8}[(F_{XX}-F_{YY}-2G_{XY})+i(G_{XX}-G_{YY}+2F_{XY})],\\
&L_{21}=\frac{1}{16}[(F_{XXX}\!+\!F_{XYY}\!+\!G_{XXY}\!+\!G_{YYY})\!+\!i(G_{XXX}\!+\!G_{XYY}\!-\!F_{XXY}\!-\!F_{YYY})].
\end{split}
\end{equation}

Summarizing the discussion above, we arrive at the following conclusion in the form of a theorem:

\begin{thm}\label{bifurcation}
Let \( \widetilde{E} = (\widetilde{u}, \widetilde{v}) \) be a unique positive fixed point of (\ref{h1}) or a positive fixed point \( E_{-} = (u_{-}, v_{-}) \) of (\ref{h2}). Assume that the parameters satisfy the conditions for the existence of the fixed point \( \widetilde{E} = (\widetilde{u}, \widetilde{v}) \), and let \( \theta = \theta_0 \) be a solution to \( q(\widetilde{u}) = 1 \).

If the parameter \( \theta \) varies in a small neighborhood of \( \theta_0 \), then the system (\ref{h12}) undergoes a Neimark-Sacker bifurcation at the fixed point \( \widetilde{E} = (\widetilde{u}, \widetilde{v}) \). Furthermore, if \( \mathcal{L} < 0 \) (respectively, \( \mathcal{L} > 0 \)), an attracting (respectively, repelling) invariant closed curve bifurcates from the fixed point for \( \theta < \theta_0 \) (respectively, \( \theta > \theta_0 \)).
\end{thm}

\section{Numerical simulations}

\subsection{Bifurcation of the model (\ref{h1}) with given parameter values.}

Consider the system (\ref{h1}) with parameters \( c = 2 \), \( \beta = 2 \), and \( r = 0.5 \). Then, to solve the equation \( q(\overline{u}) = 1 \) with respect to \( \theta \), we obtain the following equation for $x>0$:

\[
2 - \frac{8}{9 - x - \sqrt{(1-x)^2 + 8}} = \frac{16 x}{\left( 3 + x + \sqrt{(1-x)^2 + 8} \right)^2}.
\]

From this, we get that \( \theta = \theta_0 \approx 1.2012 \), and the corresponding fixed point is \(
\overline{E} \approx (0.38, 0.62).\)

The corresponding multipliers are

\[
\lambda_1 \approx 0.81 - 0.5864i, \quad \lambda_2 \approx 0.81 + 0.5864i.
\]

Moreover, we verify that

\[
|\lambda_{1,2}| = 1, \quad \frac{d |\lambda_{1,2}|}{d\theta^*} \Bigm|_{\theta^* = 0} \approx -0.10496 < 0.
\]

The relevant coefficients are computed as

\[
L_{20} \approx 0.0225 + 0.1809i, \quad L_{11} \approx -0.1587 - 0.1012i,
\]
\[
L_{02} \approx -0.4197 + 0.0142i, \quad L_{21} \approx 0.0209 - 0.049i.
\]

Additionally, we obtain

\[
\mathcal{L} \approx -0.132 < 0.
\]

Thus, according to Theorem \ref{bifurcation}, an attracting (since $\mathcal{L}<0$ ) invariant closed curve bifurcates from the fixed point for \( \theta < \theta_0 \).

In Figure \ref{fig1} (a), the fixed point remains attracting since \( q(\overline{u}) < 1 \). In Figures \ref{fig1} (b), (c), and (d), an invariant closed curve is formed as \( q(\overline{u}) > 1 \), and furthermore, this curve is attracting. In Figure \ref{diag1}, the bifurcation diagram for each variable \( u \) and \( v \) is presented as \( \theta \) varies in the range \( [0.1, 5] \). It can be observed that the system exhibits stable behavior for \( \theta > \theta_0 \approx 1.2012 \).

\subsection{Bifurcation of the model (\ref{h2}) with given parameter values.}

Consider the system (\ref{h2}) with parameters \( c = 0.25 \), \( \beta = 2 \), and \( r = 0.5 \). Then, by solving the equation \( q(u_{-}) = 1 \) with respect to \( \theta \), we have that \( \theta = \theta_0 \approx 1.035 \), and the corresponding fixed point is
\(E_{-} \approx (0.292, 0.708).\)

The corresponding multipliers are

\[
\lambda_1 \approx 0.854 - 0.520i, \quad \lambda_2 \approx 0.854 + 0.520i.
\]

Moreover, we verify that

\[
|\lambda_{1,2}| = 1, \quad \frac{d |\lambda_{1,2}|}{d\theta^*} \Bigm|_{\theta^* = 0} \approx -0.1861 < 0.
\]

The relevant coefficients are computed as

\[
L_{20} \approx -0.0067 + 0.2255i, \quad L_{11} \approx -0.1323 - 0.0080i,
\]
\[
L_{02} \approx -0.3308 - 0.0375i, \quad L_{21} \approx 0.0342 - 0.0171i.
\]

Additionally, we obtain

\[
\mathcal{L} \approx -0.0328 < 0.
\]

Thus, based on Theorem \ref{bifurcation}, an attracting invariant closed curve emerges from the fixed point for \( \theta < \theta_0 \).

In Figure \ref{fig3} (a), the fixed point remains stable since \( q(u_{-}) < 1 \). In Figures \ref{fig3} (b), (c), and (d), an invariant closed curve is formed as \( q(u_{-}) > 1 \), and additionally, this curve is attractive. In Figure \ref{diag2}, the bifurcation diagram for each variable \( u \) and \( v \) is shown as \( \theta \) changes within the interval \( [0.1, 3] \). It can be seen that the system demonstrates stable behavior for \( \theta > \theta_0 \approx 1.035 \).

\section*{Conclusion}
In this study, we explore the dynamics of a phytoplankton-zooplankton system, where zooplankton feed on phytoplankton according to a linear functional response, while the toxin released by phytoplankton follows a Holling-type saturation mechanism. We analyze two common ecological scenarios--Holling type II and Holling type III functional responses for toxin distribution.

Our findings reveal that when toxin distribution follows a Holling type II response, the system stabilizes around a single equilibrium point, indicating a predictable balance between phytoplankton and zooplankton populations. However, under Holling type III dynamics, the interaction becomes more intricate, with the potential for up to three equilibrium states. In cases where two positive fixed points exist, one is always a saddle point; when three exist, one remains a saddle, and the other is attracting. This highlights how varying toxin levels can lead to multiple stable and unstable population states, potentially triggering abrupt ecological shifts.

Furthermore, our analysis uncovers the role of a Neimark-Sacker bifurcation, which can lead to oscillatory population dynamics. This suggests that, under certain conditions, the system may transition from a stable equilibrium to sustained population cycles--phenomena commonly observed in real aquatic ecosystems. To substantiate our theoretical results, we select specific parameters and identify the bifurcation parameter, alongside other critical quantities, demonstrating the occurrence of the Neimark-Sacker bifurcation and the formation of invariant closed curves. Additionally, we present bifurcation diagrams, maximum Lyapunov exponents, and trajectories for given parameters and initial conditions, providing further insight into the system's behavior. These results significantly enhance our understanding of how toxin-mediated interactions shape predator-prey dynamics in plankton communities and their potential ecological implications.

\section*{Appendix}

\begin{lemma}[Lemma 2.1, \cite{Cheng}]\label{lem1}
Let \( F(\lambda) = \lambda^2 + B \lambda + C, \) where \( B \) and \( C \) are real constants. Suppose \( \lambda_1 \) and \( \lambda_2 \) are the roots of the equation \( F(\lambda) = 0. \) Then the following statements hold:
\begin{itemize}
    \item[(i)] If \( F(1) > 0, \) then:
    \begin{itemize}
        \item[(i.1)] \( |\lambda_1| < 1 \) and \( |\lambda_2| < 1 \) if and only if \( F(-1) > 0 \) and \( C < 1; \)
        \item[(i.2)] \( \lambda_1 = -1 \) and \( \lambda_2 \neq -1 \) if and only if \( F(-1) = 0 \) and \( B \neq 2; \)
        \item[(i.3)] \( |\lambda_1| < 1 \) and \( |\lambda_2| > 1 \) if and only if \( F(-1) < 0; \)
        \item[(i.4)] \( |\lambda_1| > 1 \) and \( |\lambda_2| > 1 \) if and only if \( F(-1) > 0 \) and \( C > 1; \)
        \item[(i.5)] \( \lambda_1 \) and \( \lambda_2 \) are a pair of conjugate complex roots and \( |\lambda_1| = |\lambda_2| = 1 \) if and only if \( -2 < B < 2 \) and \( C = 1; \)
        \item[(i.6)] \( \lambda_1 = \lambda_2 = -1 \) if and only if \( F(-1) = 0 \) and \( B = 2. \)
    \end{itemize}

    \item[(ii)] If \( F(1) = 0, \) i.e., 1 is a root of \( F(\lambda) = 0, \) then the other root \( \lambda \) satisfies
    \[
    |\lambda| \begin{cases}
    < 1, & \text{if } |C| < 1, \\
    = 1, & \text{if } |C| = 1, \\
    > 1, & \text{if } |C| > 1.
    \end{cases}
    \]

    \item[(iii)] If \( F(1) < 0, \) then \( F(\lambda) = 0 \) has one root in the interval \( (1, \infty). \) Moreover:
    \begin{itemize}
        \item[(iii.1)] The other root \( \lambda \) satisfies \( \lambda < (=) -1 \) if and only if \( F(-1) < (=) 0; \)
        \item[(iii.2)] The other root \( \lambda \) satisfies \( -1 < \lambda < 1 \) if and only if \( F(-1) > 0. \)
    \end{itemize}
\end{itemize}
\end{lemma}

\begin{defn}\label{def1}
Let \( E(x, y) \) be a fixed point of the operator \( F: \mathbb{R}^2 \to \mathbb{R}^2 \), and let \( \lambda_1, \lambda_2 \) be the eigenvalues of the Jacobian matrix \( J = J_F \) at the point \( E(x, y) \).

\begin{enumerate}
    \item[(i)] If \( |\lambda_1| < 1 \) and \( |\lambda_2| < 1 \), then the fixed point \( E(x, y) \) is called an \textbf{attractive}.

    \item[(ii)] If \( |\lambda_1| > 1 \) and \( |\lambda_2| > 1 \), then the fixed point \( E(x, y) \) is called \textbf{repelling}.

    \item[(iii)] If \( |\lambda_1| < 1 \) and \( |\lambda_2| > 1 \) (or \( |\lambda_1| > 1 \) and \( |\lambda_2| < 1 \)), then the fixed point \( E(x, y) \) is called a \textbf{saddle}.

    \item[(iv)] If either \( |\lambda_1| = 1 \) or \( |\lambda_2| = 1 \), then the fixed point \( E(x, y) \) is called \textbf{non-hyperbolic}.
\end{enumerate}
\end{defn}

\section*{Declarations}

\textbf{Ethical approval:} The author confirms that this study does not involve human participants or animals, and no ethical approval was required.

\textbf{Funding:} No funding was received for conducting this study.

\textbf{Conflict of interest:} The author declares that there are no conflicts of interest regarding the publication of this paper.

\begin{figure}[h!]
    \centering
    \subfigure[\tiny$\theta=1.205, \overline{u}\approx0.3801, q(\overline{u})\approx0.9995, (0.2, 1.1).$]{\includegraphics[width=0.45\textwidth]{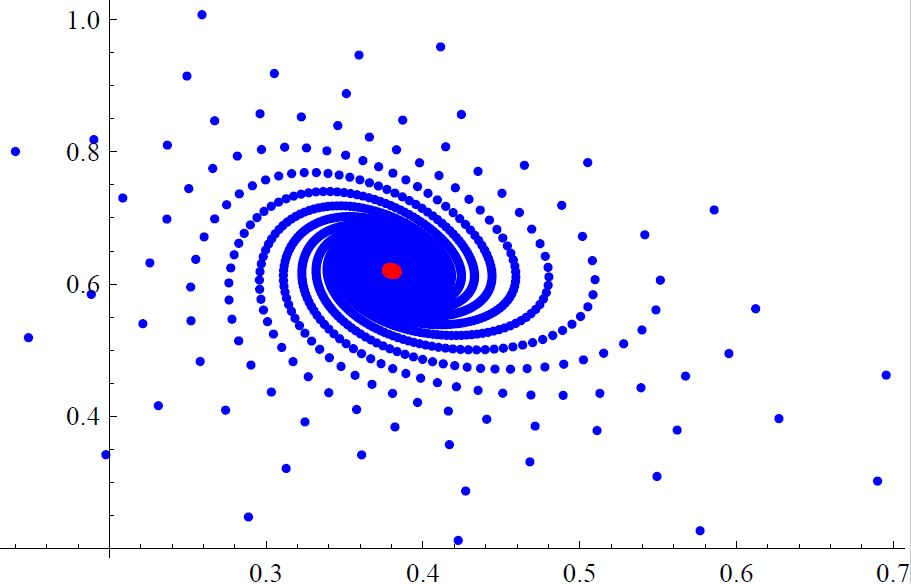}} \hspace{0.3in}
    \subfigure[\tiny$\theta=1.201, \overline{u}\approx0.37957, q(\overline{u})\approx1.00003, (0.2, 1.1).$]{\includegraphics[width=0.45\textwidth]{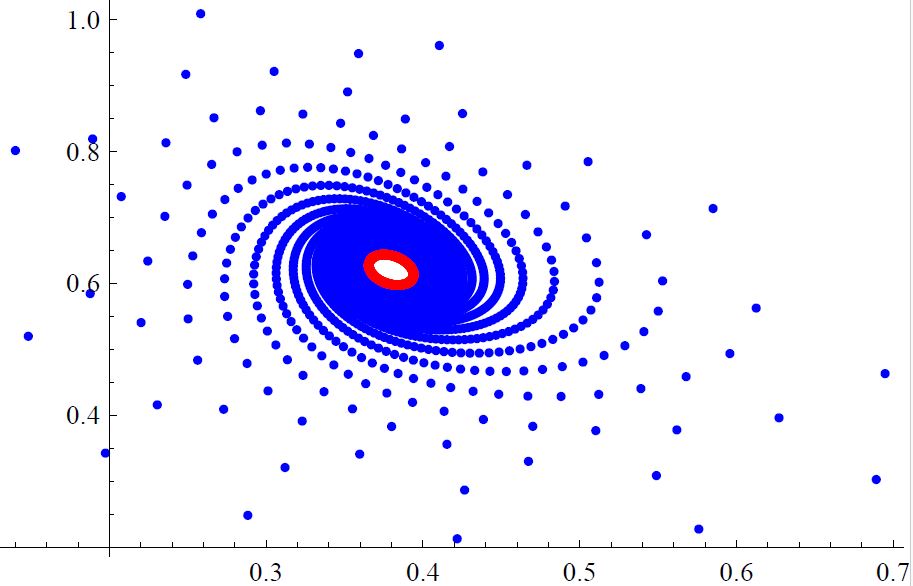}}
    \subfigure[\tiny$\theta=1.12, \overline{u}\approx0.368871, q(\overline{u})\approx1.01039, (0.2, 1.1).$]{\includegraphics[width=0.45\textwidth]{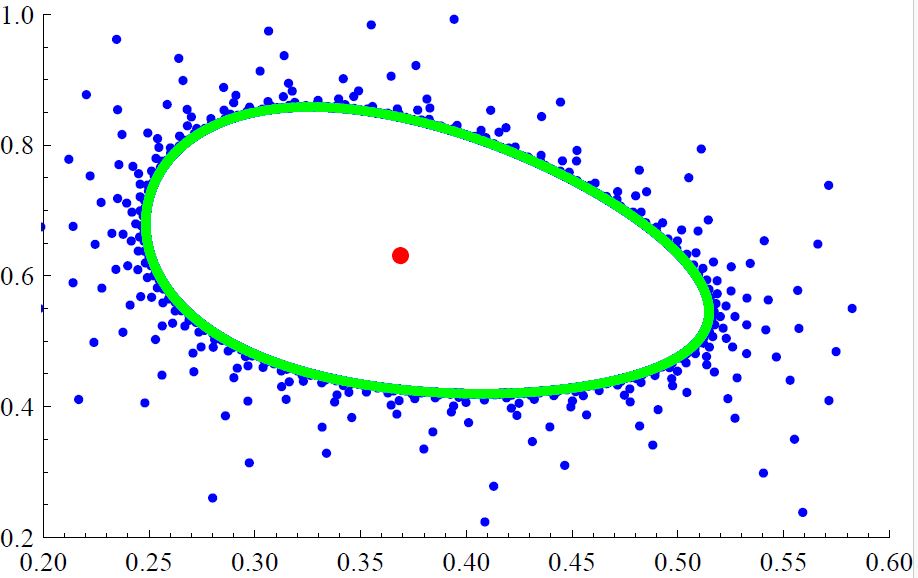}} \hspace{0.3in}
    \subfigure[\tiny$\theta=1.12, \overline{u}\approx0.368871, q(\overline{u})\approx1.01039, (0.35, 0.6).$]{\includegraphics[width=0.45\textwidth]{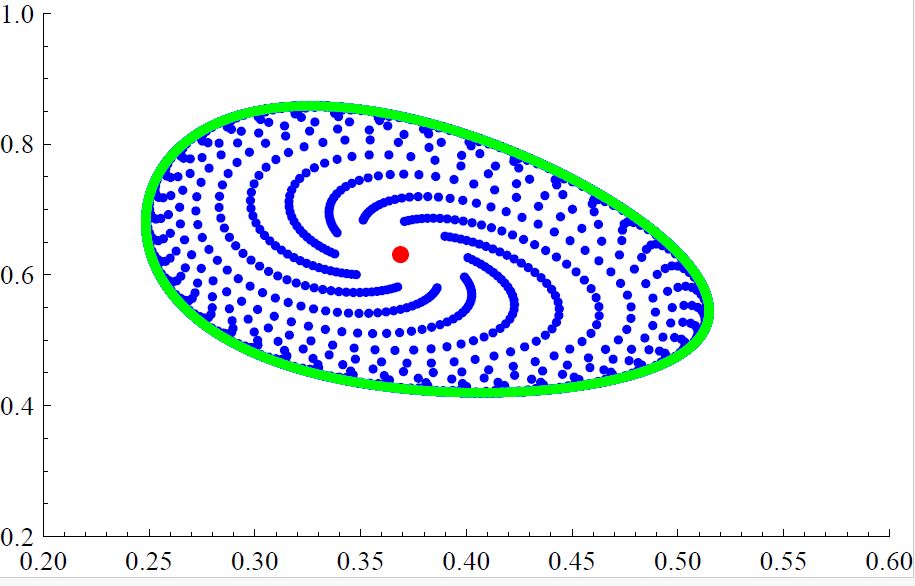}}\hspace{0.3in}
     \subfigure[\tiny$\theta=0.7, (u^0, v^0)=(0.2, 1.1).$]{\includegraphics[width=0.45\textwidth]{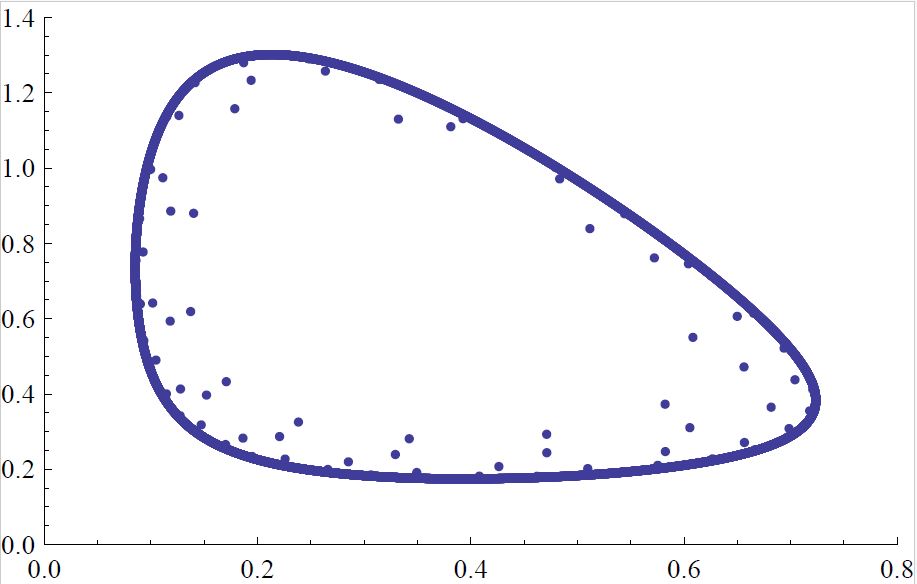}}\hspace{0.3in}
      \subfigure[\tiny$\theta=0.1, (u^0, v^0)=(0.2, 1.1).$]{\includegraphics[width=0.45\textwidth]{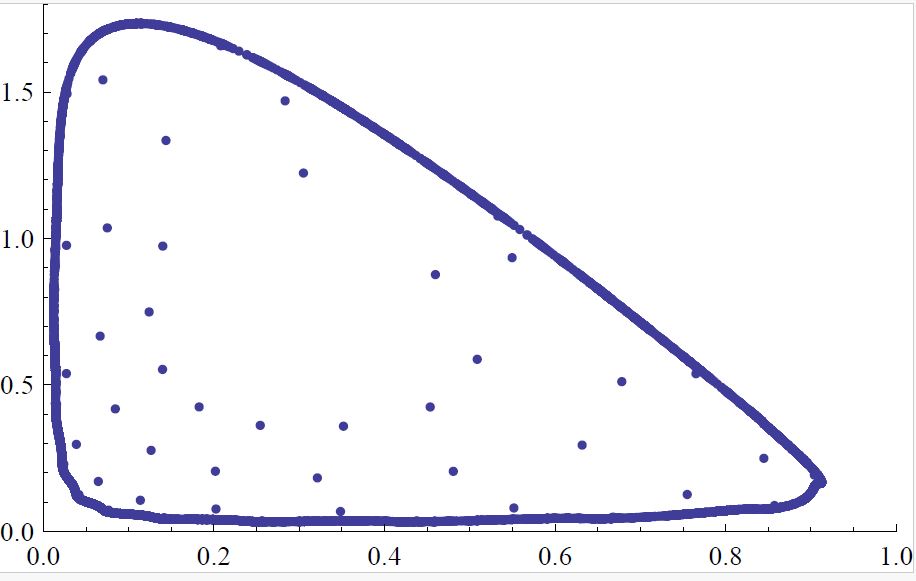}}
    \caption{Phase portraits for the system (\ref{h1}) with parameters \( c =0.25,  \beta = 2 \), \( r = 0.5 \), and \( n = 10,000 \). The red point represents the fixed point \( E_{-} \), while the green curve indicates an invariant closed curve that is attracting. In panel (c), the initial point is taken from outside the invariant closed curve, while in panel (d), the initial point is taken from inside the invariant closed curve. In panels (e) and (f), the closed curve expands and undergoes a transformation in shape.}
    \label{fig1}
\end{figure}

\begin{figure}[h!]
    \centering
    \subfigure[]{\includegraphics[width=0.45\textwidth]{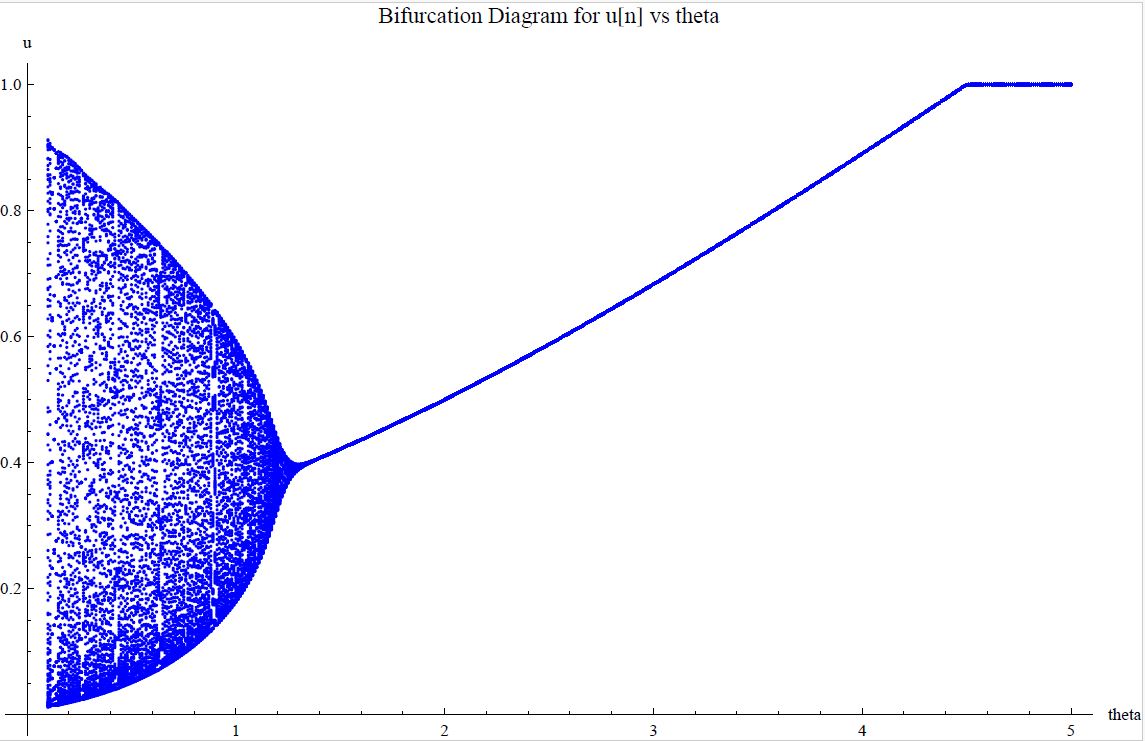}} \hspace{0.3in}
    \subfigure[]{\includegraphics[width=0.45\textwidth]{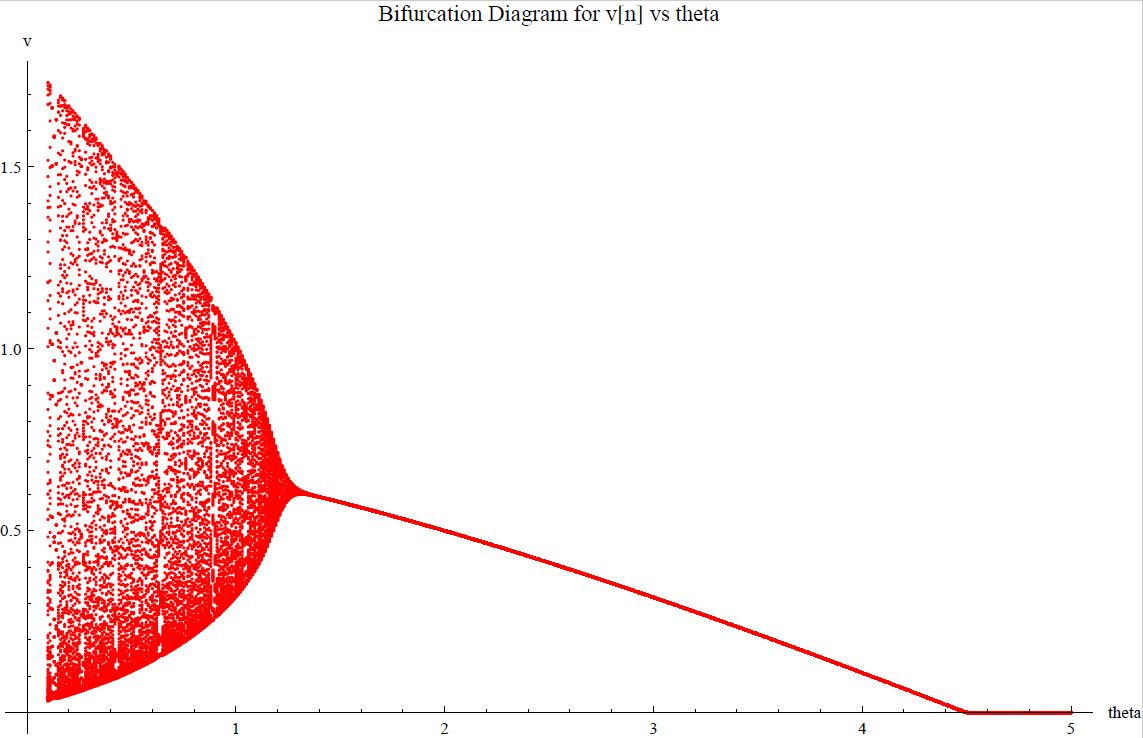}}
    \subfigure[]{\includegraphics[width=0.7\textwidth]{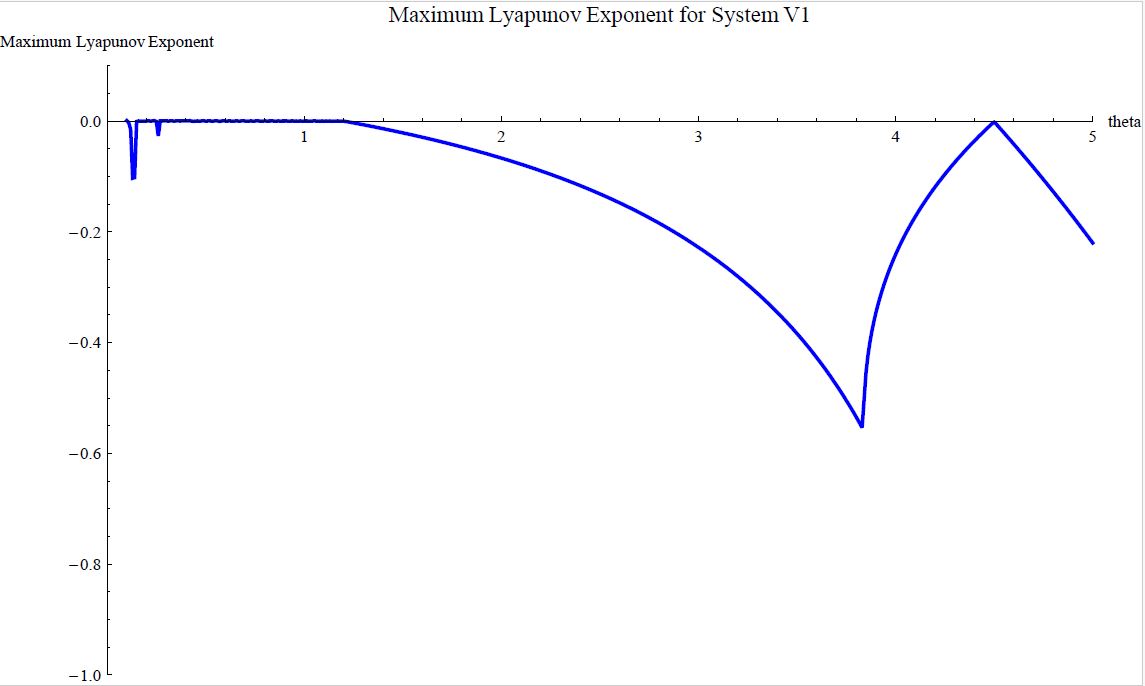}}
    \caption{Bifurcation diagrams for the system (\ref{h1}) with parameters \( r = 0.5 \), \( c = \beta = 2 \), and initial values \( u^0 = 0.2 \), \( v^0 = 1.1 \), as the bifurcation parameter \( \theta \) varies in the interval \( 0.1 \leq \gamma \leq 5 \). Panel (c) presents the maximum Lyapunov exponents corresponding to panels (a) and (b).}
    \label{diag1}
\end{figure}

\begin{figure}[h!]
    \centering
    \subfigure[\tiny$\theta=1.1, u_{-}\approx0.2977, \overline{q}(u_{-})\approx0.9894, (0.2, 1.1).$]{\includegraphics[width=0.45\textwidth]{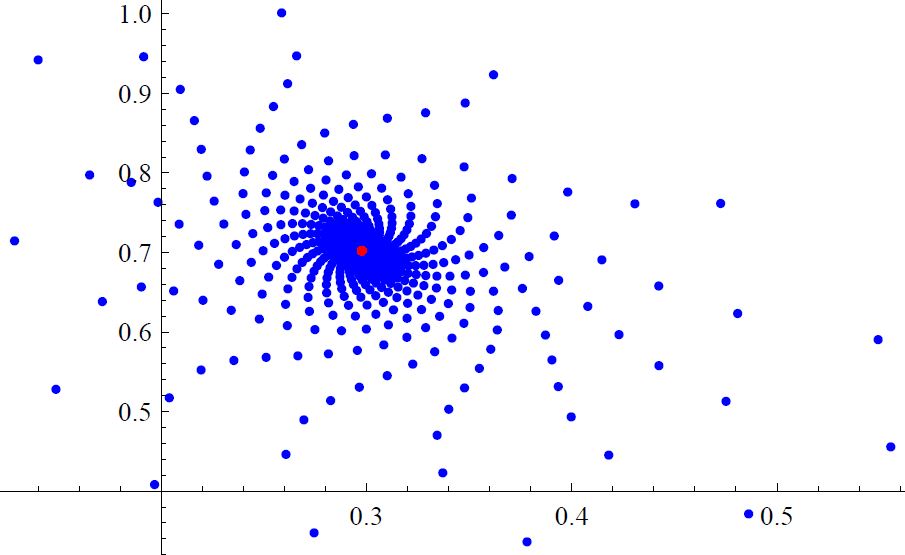}} \hspace{0.3in}
    \subfigure[\tiny$\theta=1.03, u_{-}\approx0.2934, \overline{q}(u_{-})\approx1.00115, (0.2, 1.1).$]{\includegraphics[width=0.45\textwidth]{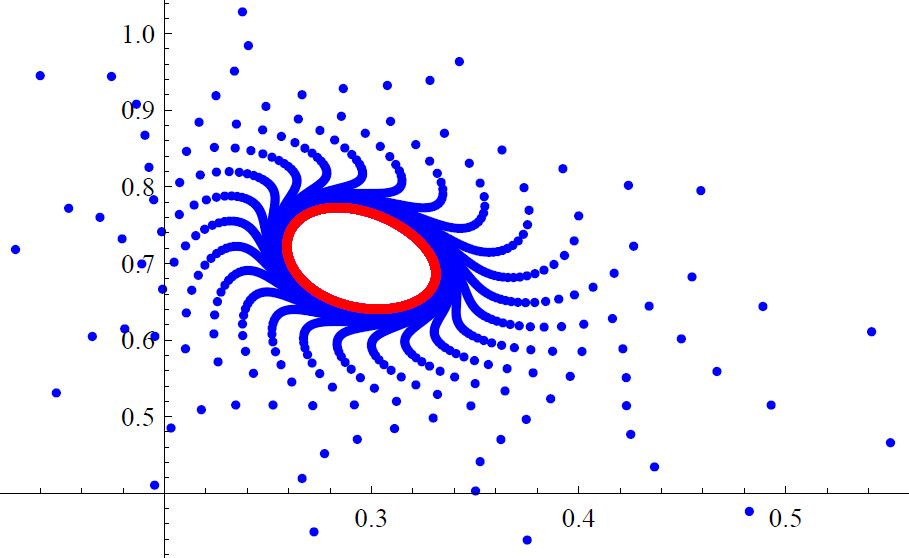}}
    \subfigure[\tiny$\theta=1.01, u_{-}\approx0.292, \overline{q}(u_{-})\approx1.0044, (0.2, 1.1).$]{\includegraphics[width=0.45\textwidth]{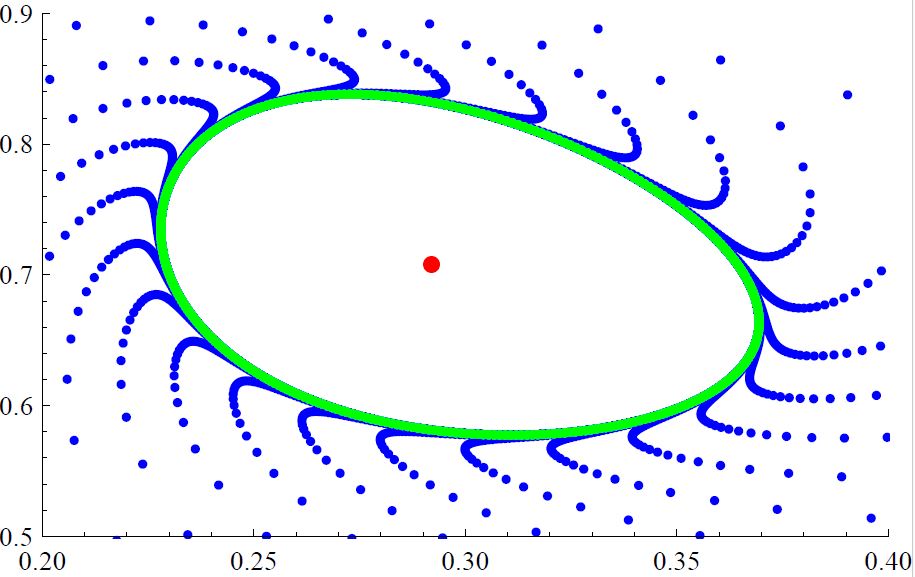}} \hspace{0.3in}
    \subfigure[\tiny$\theta=1.01, u_{-}\approx0.292, \overline{q}(u_{-})\approx1.0044, (0.29, 0.69).$]{\includegraphics[width=0.45\textwidth]{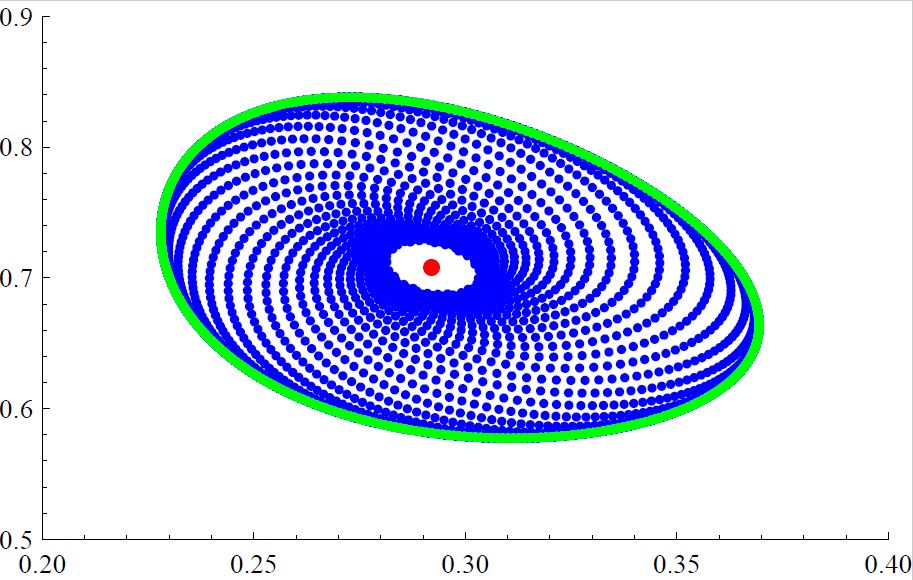}}\hspace{0.3in}
    \subfigure[\tiny$\theta=0.7, (u^0, v^0)=(0.2, 0.7).$]{\includegraphics[width=0.45\textwidth]{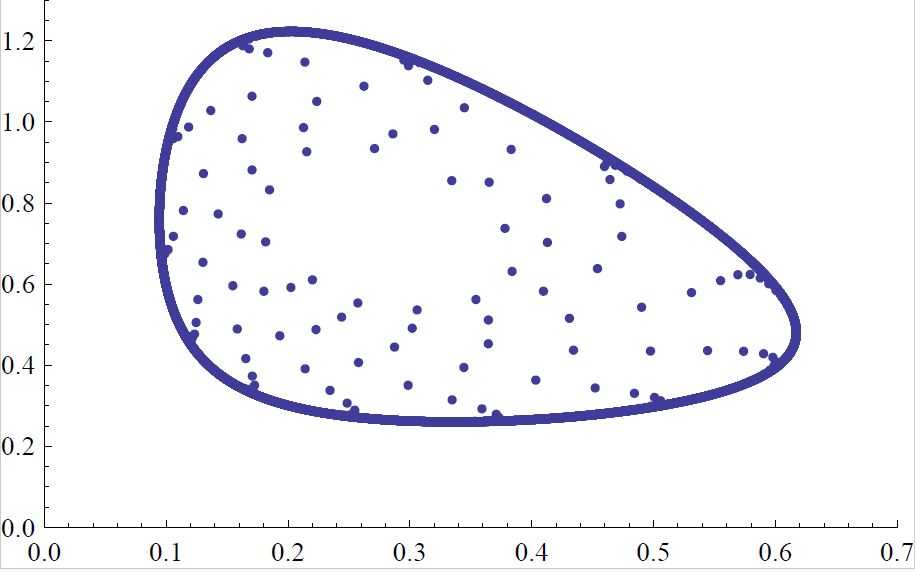}}\hspace{0.3in}
    \subfigure[\tiny$\theta=0.1, (u^0, v^0)=(0.2, 0.7).$]{\includegraphics[width=0.45\textwidth]{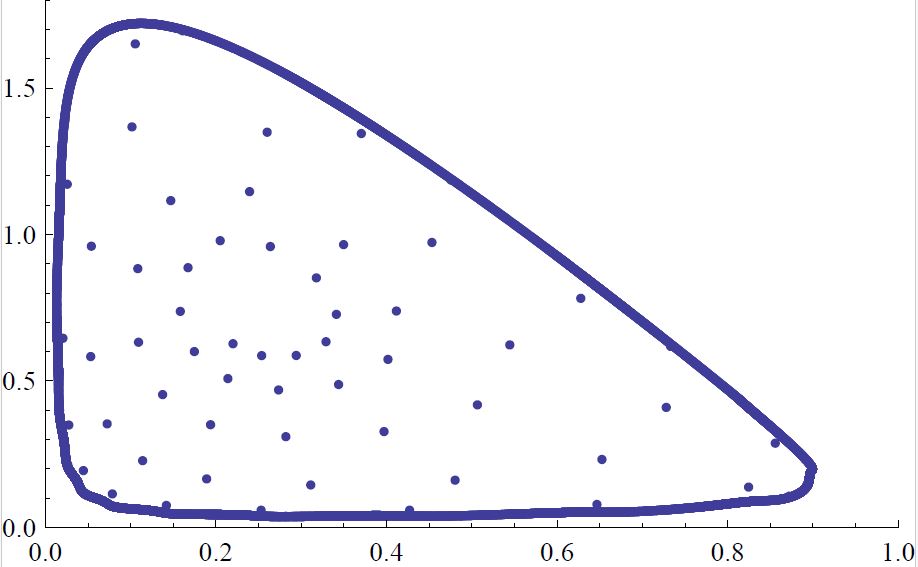}}
    \caption{Phase portraits for the system (\ref{h2}) with parameters \( c = 0.25 \), \( \beta = 2 \), \( r = 0.5 \), and \( n = 10,000 \). The red point represents the fixed point \( E_{-} \), while the green curve indicates an invariant closed curve that is attracting. In panel (c), the initial point is taken from outside the invariant closed curve, while in panel (d), the initial point is taken from inside the invariant closed curve. In panels (e) and (f), the closed curve expands and undergoes a transformation in shape.}
    \label{fig3}
\end{figure}

\begin{figure}[h!]
    \centering
    \subfigure[]{\includegraphics[width=0.45\textwidth]{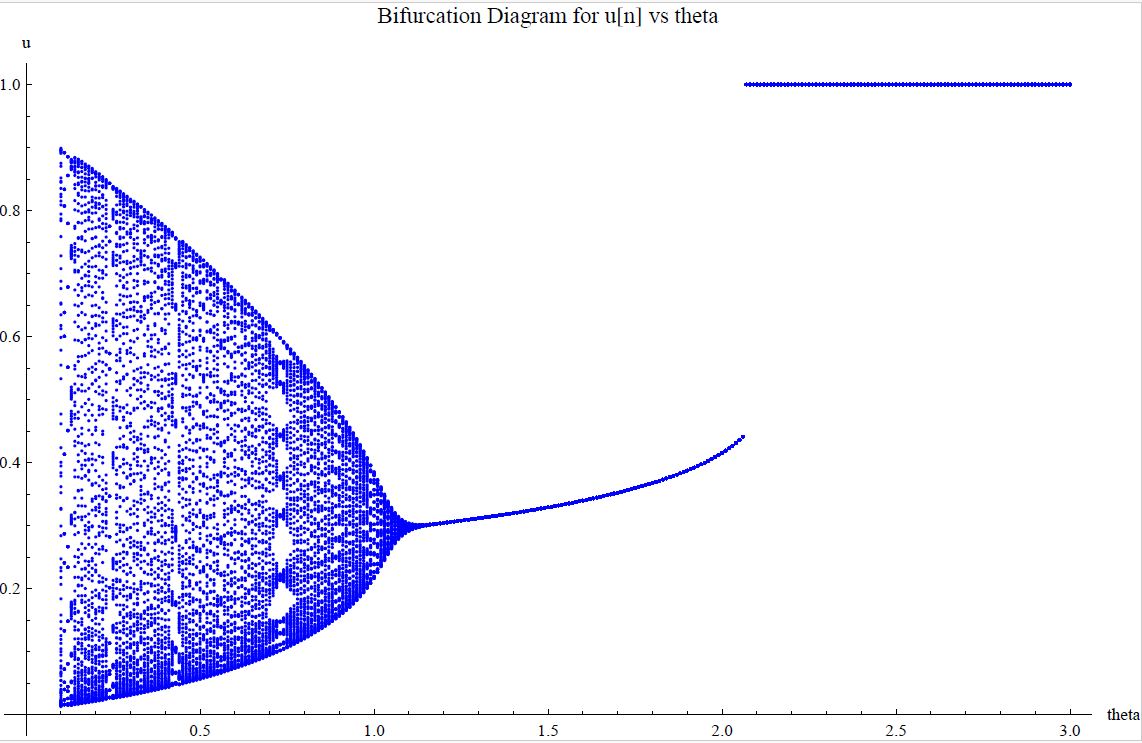}} \hspace{0.3in}
    \subfigure[]{\includegraphics[width=0.45\textwidth]{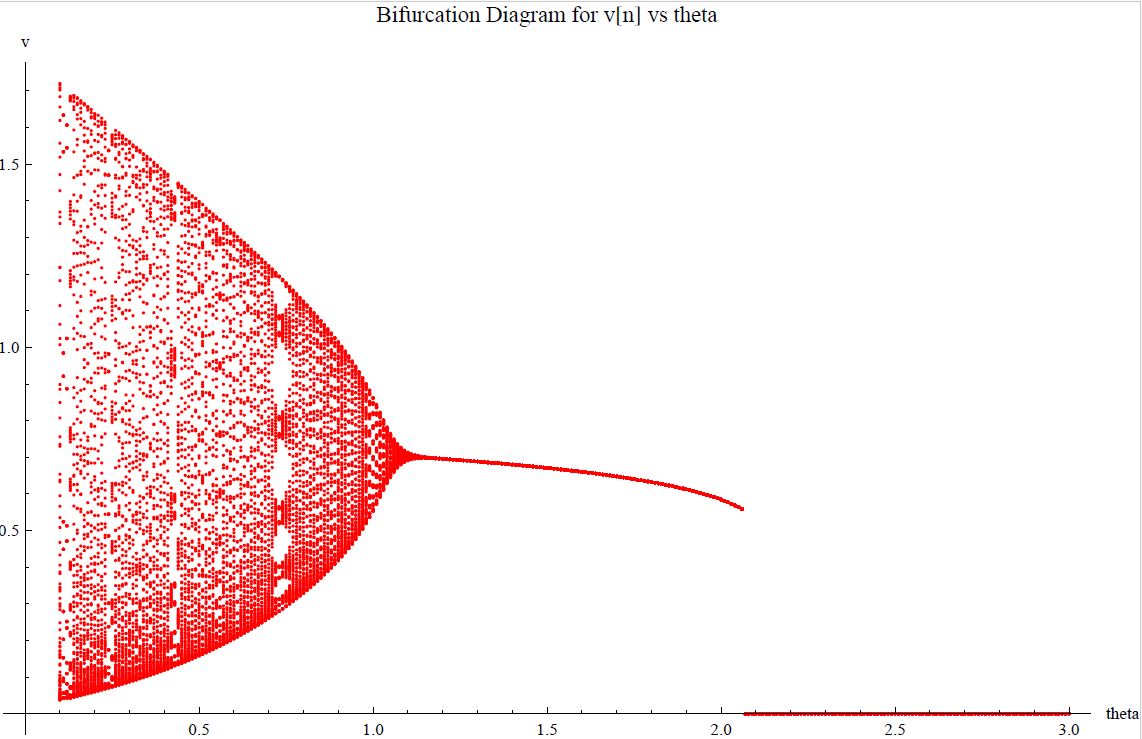}}
    \subfigure[]{\includegraphics[width=0.7\textwidth]{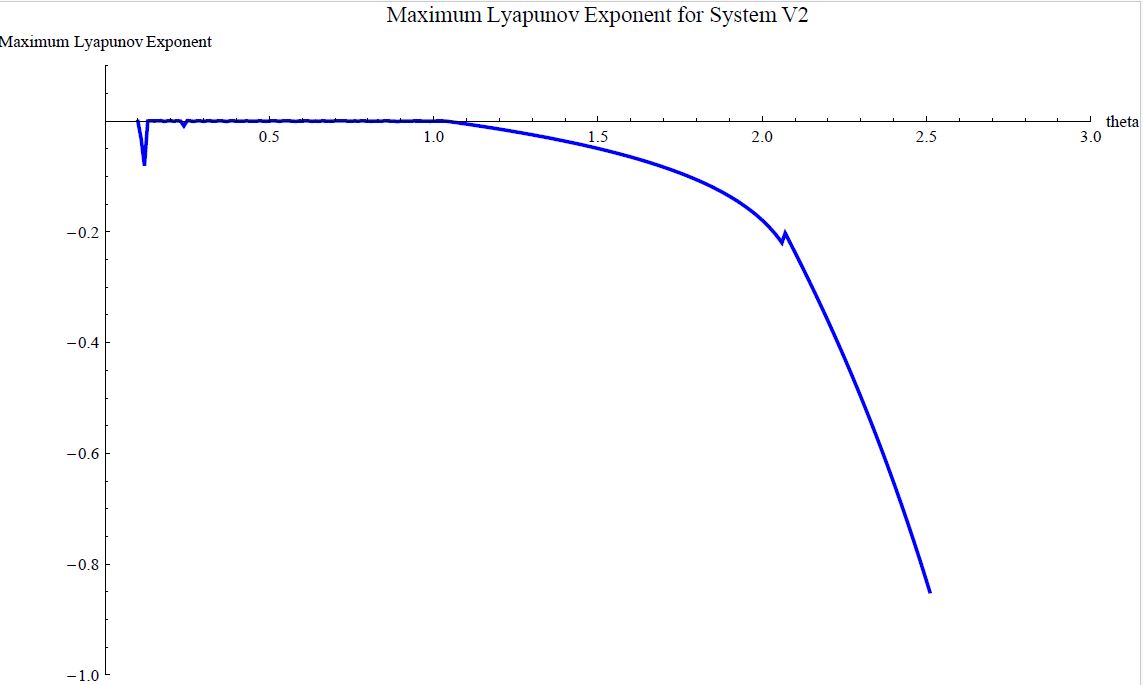}}
\caption{Bifurcation diagrams for the system (\ref{h2}) with parameters \( r = 0.5 \), \( c = 0.25 \), \( \beta = 2 \), and initial values \( u^0 = 0.2 \), \( v^0 = 1.1 \), as the bifurcation parameter \( \theta \) varies in the interval \( 0.1 \leq \gamma \leq 3 \). Panel (c) presents the maximum Lyapunov exponents corresponding to panels (a) and (b).}
    \label{diag2}
\end{figure}

\end{document}